\documentclass{aptpub}

\usepackage{graphicx}
\usepackage{verbatim}

\authornames{Mariana Olvera-Cravioto} 
\shorttitle{Weighted Random Sums} 

\newtheorem{mydef}[thm]{\noindent Definition}
\newtheorem{myexample}[thm]{\noindent Example}
\def\Indicator{\mathop{\hskip0pt{1}}\nolimits}

\renewcommand{\crnotice}{}

\begin{document}

\title{Asymptotics for Weighted Random Sums}

\authorone[Columbia University]{Mariana Olvera-Cravioto}

\addressone{Department of Industrial Engineering and Operations Research, Columbia University, New York, NY 10027 }

\begin{abstract}
Let $\{ X_i\}$ be a sequence of independent identically distributed random variables with an intermediate regularly varying (IR) right tail $\overline{F}$. Let $(N, C_1, C_2, \dots)$ be a nonnegative random vector independent of the $\{X_i\}$ with $N \in \mathbb{N} \cup \{\infty\}$. We study the weighted random sum  $S_N = \sum_{i=1}^N C_i X_i$, and its maximum, $M_N = \sup_{1 \leq k < N+1} \sum_{i=1}^k C_i X_i$. This type of sums appear in the analysis of stochastic recursions, including weighted branching processes and autoregressive processes. In particular, we derive conditions under which
$$P(M_N > x) \sim P(S_N > x) \sim E\left[ \sum_{i=1}^N \overline{F}(x/C_i) \right],$$
as $x \to \infty$. When $E[X_1] > 0$ and the distribution of $Z_N = \sum_{i=1}^N C_i$ is also IR, we obtain the asymptotics 
$$P(M_N > x) \sim P(S_N > x) \sim E\left[ \sum_{i=1}^N \overline{F}(x/C_i) \right] + P(Z_N > x/E[X_1]).$$
For completeness, when the distribution of $Z_N$ is IR and heavier than $\overline{F}$, we also obtain conditions under which the asymptotic relations
$$P(M_N > x) \sim P(S_N > x) \sim P(Z_N > x/E[X_1])$$
hold. 
\end{abstract}

\keywords{Randomly weighted sums; randomly stopped sums; heavy tails; intermediate regular variation; regular variation; Breiman's theorem}

\ams{60G50}{60F10, 60J80, 60G70}

\section{Introduction}

The analysis of randomly weighted sums plays an important role in the insurance and economic literature. A well known example in ruin theory interprets the weights as discount factors and the sequence $\{X_i\}$ as the net losses of an insurance company to analyze the probability of ruin either in finite or infinite time (see, e.g., \cite{Tan_Tsi_03}). In economics, the $\{X_i\}$ can be interpreted as net incomes of an investment and the weights as random return rates (see, e.g., \cite{Goovaerts_etal_05}).  In general, randomly weighted sums appear in the analysis of random stochastic equations (e.g., autoregressive processes), and have applications in many areas beyond the ones mentioned above. If we further assume that the number of terms in the sum can be random, we obtain a randomly stopped and randomly weighted sum. Such {\em weighted random sums} appear in the context of weighted branching processes and fixed-point equations of smoothing transforms (see \cite{Liu_98, Iksanov_04, Alsm_Mein_10b}), and more recently, in the analysis of information ranking algorithms, e.g., Google's PageRank (see \cite{Volk_Litv_08, Jel_Olv_10}). In all the examples mentioned above, the $\{X_i\}$ are often assumed to be heavy-tailed. Hence, the results in this paper combine two different topics in the literature for sums of heavy-tailed random variables, the analysis of randomly weighted sums and the analysis of randomly stopped sums.

Consider a sequence $\{ X_i \}_{i \geq 1}$ of independent, identically distributed (i.i.d.) random variables with finite mean and a heavy right tail distribution $\overline{F}$, where by heavy we mean $E[e^{\epsilon X_1^+} ] = \infty$ for all $\epsilon > 0$, and $x^+ = \max\{x, 0\}$. Let $(N, C_1, C_2, \dots)$ be a nonnegative random vector independent of the $\{X_i\}$ with $N \in \mathbb{N} \cup \{\infty\}$.  We study the asymptotic behavior of the randomly weighted and randomly stopped sum $\sum_{i=1}^N C_i X_i$, and of its maximum, $\sup_{1 \leq k < N+1} \sum_{i=1}^k C_i X_i$; the weights $\{C_i\}$ are allowed to be arbitrarily dependent and may depend on $N$ as well, and the convention throughout the paper is that $N +1 = \infty$ if $N = \infty$.  We point out that it is possible to avoid the introduction of $N$ by redefining the weights $\widetilde{C}_i = C_i \Indicator(i \leq N)$ and considering the sum $\sum_{i=1}^\infty \widetilde{C}_i X_i$,  but to emphasize the possibility of having a random number of summands we choose to keep the results in this paper in terms of $N$. Throughout the paper we use $f(x) \sim g(x)$ as $x \to \infty$ to denote $\lim_{x \to \infty} f(x)/g(x) = 1$, and $f(x) \asymp g(x)$ as $x \to \infty$ to denote $f(x) = O(g(x))$ and $g(x) = O(f(x))$. 

Although the literature of both weighted random sums and randomly stopped sums is extensive, this is the first paper, to our knowledge, to combine the two, and in doing so, to obtain the $N = \infty$ case under conditions that are close to the best possible. The main results also include an analysis of the cases where the asymptotic behavior of the weighted random sum does not follow the so-called {\em one-big-jump principle} ($P(\sum_{i=1}^n X_i > x) \sim n \overline{F}(x)$ as $x \to \infty$), and instead is dominated by the sum of the weights, which until now had only been done for the special case $C_i \equiv 1$ (see \cite{Jess_Miko_06, Den_Foss_Kor_10}).  

To gain some insight into the asymptotics
\begin{equation} \label{eq:IntroOnlySummands}
P\left( \sup_{1\leq k < N+1} \sum_{i=1}^k C_i X_i > x \right) \sim P\left( \sum_{i=1}^N C_i X_i > x \right) \sim E\left[ \sum_{i=1}^N \overline{F}(x/ C_i) \right], \qquad x \to \infty,
\end{equation}
note that if the $\{X_i\}$ are i.i.d. and heavy-tailed, and the weights $\{C_i\}$ satisfy suitable conditions, then the random variables $\{C_i X_i\}_{i \geq 1}$ behave as if they were independent, and the one-big-jump principle gives \eqref{eq:IntroOnlySummands}.  The asymptotic relation
$$P\left( \sum_{i=1}^\infty C_i X_i > x \right) \sim E\left[ \sum_{i=1}^\infty C_i^\alpha \right] \overline{F}(x), \qquad x \to \infty,$$
was established in \cite{Res_Will_91} for nonnegative and regularly varying $\{X_i\}$ (denoted $\{ X_i\}$ in $\mathcal{R}_{-\alpha}$, $\alpha > 0$), and \eqref{eq:IntroOnlySummands} was proven in \cite{Goovaerts_etal_05} for real-valued $\{X_i\}$ with regularly varying right tail and deterministic $N$, either $N = n$ (finite) or $N = \infty$. The setting where the $\{X_i\}$ are real-valued with right tail in the extended regular variation class was studied in \cite{Wang_Tang_06} ($N = n$) and \cite{Zhang_Shen_Weng_09} (both $N = n$ and $N = \infty$); in the latter the $\{X_i\}$ are allowed to be generally dependent with no bivariate upper tail dependence. Deterministic, real-valued weights with the $\{X_i\}$ in $\mathcal{R}_{-\alpha}$ were considered in \cite{Mik_Sam_00}. We point out that in all the mentioned works where $N = \infty$, the conditions imposed on the weights are considerably stronger than those imposed for a finite number of terms.   The first result in this paper establishes \eqref{eq:IntroOnlySummands} for i.i.d., real-valued $\{X_i\}$ with finite mean, right tail in the intermediate regular variation class, and $N$ potentially random; the conditions on the weights are basically the same regardless of whether $N$ is deterministic, random, or infinity.  Results for more general classes of heavy-tailed distributions but stronger conditions on the weights and $N = n$ are given in \cite{Tan_Tsi_03} (for bounded weights) and \cite{Chen_Su_06} (for $C_i = \prod_{j=1}^i Y_j$ and $\{Y_j\} \geq 0$ i.i.d. from a specific class of distributions).  The finite mean restriction is due to our interest in analyzing the asymptotic behavior of the randomly weighted and randomly stopped sum when it is not solely determined by the one-big-jump principle. 

As mentioned earlier, the scope of this paper is to combine the analysis of randomly weighted sums with that of randomly stopped sums. For instance, if we set $C_i \equiv 1$ for all $i \geq 1$, then the subexponential asymptotics $P\left( \sum_{i=1}^n X_i > x \right) \sim n P(X_1 >x)$ is known to hold, under suitable conditions on $N$, even for a random number of summands. The asymptotic relation
\begin{equation} \label{eq:MainRandomSum}
P\left( \sum_{i=1}^N X_i > x \right) \sim E[N] \overline{F}(x), \qquad x \to \infty,
\end{equation}
has a long history (see, e.g., \cite{Asm2003}, \cite{EmKlMi1997} and the references therein), although the analysis when $N$ does not have finite exponential moments is more recent. Relation \eqref{eq:MainRandomSum} was established in \cite{Dal_Ome_Ves_07} for several different sets of conditions on $N$ and the $\{X_i\}$, including some where $N$ may be subexponential. Some results imposing no conditions on $N$ and the $\{X_i\}$ in either the regularly varying or semi-exponential classes were derived in \cite{Borov_Borov_2008}.  The most general conditions were recently derived in \cite{Den_Foss_Kor_10} for $\{X_i\}$ in the class $\mathcal{S}^*$, which includes most subexponential distributions with finite mean. Moreover, the results in \cite{Den_Foss_Kor_10} also include the case where the asymptotic behavior of the randomly stopped sum is not solely determined by the one-big-jump principle, and, in particular, it was shown that
$$P\left( \sup_{1\leq k < N+1} \sum_{i=1}^k X_i > x \right) \sim P\left( \sum_{i=1}^N X_i > x \right) \sim E[N] \overline{F}(x ) + P(N > x/E[X_1]), \qquad x \to \infty,
$$
provided that the $\{X_i\}$ belong to $\mathcal{S}^*$, $E[X_1] > 0$, and $N$ belongs to the intermediate regular variation class.  The term $P(N > x/E[X_1])$ corresponds to the situation where the asymptotic behavior of the random sum is determined by the law of large numbers. This last asymptotic relation was previously proven in \cite{Jess_Miko_06} for the case where both $N$ and $X_1$ are nonnegative and belong to $\mathcal{R}_{-\alpha}$ with $\alpha \geq 1$, $P(N > x) \sim c P(X_1 > x)$ for some constant $c > 0$, and $E[X_1] < \infty$. All the results in \cite{Den_Foss_Kor_10} are readily applicable to our randomly weighted sums setting provided that the $\{C_i\}$ are i.i.d., independent of $N$, and that the sequence $\{C_i X_i\}$ belongs to $\mathcal{S}^*$. The second result in this paper extends the analysis to allow the vector $(N, C_1, C_2, \dots)$ to have an arbitrary distribution, but restricts the $\{X_i\}$ to belong to the intermediate regular variation class. In this context, the term $P(N > x/E[X_1])$ is replaced by $P\left(\sum_{i=1}^N C_i > x/E[X_1]\right)$ . 

For completeness, the third and last result in this paper considers the case where the behavior of the randomly stopped and randomly weighted sum is completely determined by the effects of the sum $\sum_{i=1}^N C_i$, which when the weights $\{C_i\}$ are i.i.d. and independent of $N$, corresponds to the dominance of the law of large numbers. The intuition remains the same in the presence of weights, as it corresponds to the situation where all the $\{X_i\}$ behave in an ordinary way, i.e., according to their mean, and it is the sum of the weights that is unusually large. Related results to those of Theorem \ref{T.ZDominates} can be found in \cite{Jess_Miko_06} for a regularly varying number of summands, $N$, $C_i \equiv 1$, and nonnegative $\{X_i\}$ with lighter tails than $N$. 

We end this section with two potential applications. The first one concerns information ranking algorithms, such as Google's PageRank algorithm for ranking webpages in the World Wide Web (WWW). If we let $R$ denote the (scale free) rank of a randomly chosen webpage, $N$ denote the number of webpages pointing to it (in-degree), and set $C_i = c/D_i$, where $D_i$ is the number of outbound links (out-degree) of the $i$th neighboring page and $0 < c < 1$ is a predetermined constant, then it can be shown that $R$ (approximately) satisfies the stochastic fixed-point equation
\begin{equation} \label{eq:PageRank}
R \stackrel{\mathcal{D}}{=} \sum_{i=1}^N C_i R_i + (1-c),
\end{equation}
where the $\{R_i\}$ are i.i.d. copies of $R$, independent of $(N, C_1, C_2, \dots)$, and $\stackrel{\mathcal{D}}{=}$ denotes equality in distribution. In the WWW, as in many other social networks, both the in-degree $N$ and the effective out-degree $D_i$ are assumed to be regularly varying. The problem of interest is to determine the proportion of highly ranked pages, which translates into the analysis of the asymptotic behavior of $P(R > x)$. The stochastic model leading to \eqref{eq:PageRank}, for the case of i.i.d. weights $\{C_i\}$ independent of $N$, was introduced in \cite{Volk_Litv_08}, and has been studied in detail in \cite{Jel_Olv_10}. The more realistic case where the vector $(N, C_1,C_2, \dots)$ is generally correlated serves as a motivating example for the results presented here. 

The second application concerns ruin probabilities. A well known example in ruin theory  where randomly weighted sums appear is in the analysis of discrete time risk models (see, e.g., \cite{Tan_Tsi_03, Wang_Tang_06}). Let $\{D_j\}$ be a sequence of i.i.d. nonnegative random variables representing discount factors per period, and let $\{ X_i\}$ be another sequence of i.i.d. real-valued random variables, independent of the $\{D_j\}$, used to denote the per period net losses of an insurance company; in many settings the $\{X_i\}$ are assumed to have a heavy right tail. Set the weight $C_i = \prod_{j=1}^i D_j$ to be the compound discount factor for period $i$. If the insurance company starts with an initial capital $x$, then its discounted surplus after $n$ periods is given by
$$W_n = x - \sum_{i=1}^n C_i X_i, \quad n \geq 1,  \qquad W_0 = x.$$
The quantities of interest are the probabilities of ruin in finite and infinite time, given respectively by
$$P\left( \max_{1\leq k \leq n} \sum_{i=1}^k C_i X_i > x \right) \qquad \text{and} \qquad P\left( \sup_{k \geq 0}  \sum_{i=1}^k C_i X_i > x \right). $$ 

The rest of the paper is organized as follows.  Upper bounds for the maximum of the randomly weighted sum are derived in Section \ref{S.UpperBound}, and lower bounds for the randomly stopped and randomly weighted sum are derived in Section \ref{S.LowerBound}. Finally, the proofs of the main results are given in Section~\ref{S.MainProofs}.

\section{Main Results} \label{S.Main}

We start by giving some definitions needed for the statement of the main theorems.

\begin{mydef}
Let $X$ be a random variable with right tail distribution $\overline{F}(x) = P(X > x)$. We say that $\overline{F}$ belongs to the {\em intermediate regular variation} (IR) class if
$$\lim_{\delta \downarrow 0} \limsup_{x \to \infty} \frac{\overline{F}((1-\delta) x)}{\overline{F}(x)} = 1.$$
\end{mydef}

We refer the reader to Chapter 2 in \cite{BiGoTe1987} for the definitions of regular variation ($\mathcal{R}_{-\alpha}$), extended regular variation (ER), and O-regular variation (OR), that are mentioned throughout the paper. It is well known that $\mathcal{R}_{-\alpha} \subset ER \subset IR \subset OR$.

\bigskip

\begin{mydef}
Let $\overline{F}(x) = P(X > x)$, $f(x) = -\log \overline{F}(x)$, and define
\begin{align*}
f_*(\lambda) = \liminf_{x \to \infty} (f(\lambda x)-f(x)) \qquad &\text{and} \qquad f^*(\lambda) = \limsup_{x \to \infty} (f(\lambda x)-f(x)), \\
\alpha_f =  \lim_{\lambda \to \infty} \frac{f^*(\lambda)}{\log \lambda} \qquad &\text{and} \qquad \beta_f =  \lim_{\lambda \to \infty} \frac{f_*(\lambda)}{\log \lambda}.
\end{align*}
The constant $\alpha_f$ is known as the {\em lower Matuszewska index} of $f$, and $\beta_f$ is known as the {\em upper Matuszewska index} of $f$, and they satisfy $0 \leq \alpha_f \leq \beta_f \leq \infty$. 
\end{mydef}

{\bf Remark:} For the OR family, Theorem 3.4.3 in \cite{BiGoTe1987} gives $0 \leq \alpha_f \leq \beta_f < \infty$. Furthermore, the constants $(-c,-d)$ in the definition of the ER class satisfy $c \leq \alpha_f \leq \beta_f \leq d$ (see pg. 68 in \cite{BiGoTe1987}). 

\bigskip

We are now ready to state the three main theorems of this paper. The first one corresponds to the setting where the one-big-jump principe dominates the behavior of the weighted random sum and its maximum. Since the weights $\{C_i\}$ are nonnegative, we use the convention that $\overline{F}(t/C_i) = 0$ for any $t \geq 0$ if $C_i = 0$.

\begin{thm} \label{T.SummandsDominate}
Suppose $\{X_i\}$ is a sequence of i.i.d. random variables with right tail distribution $\overline{F} \in IR$, Matuszewska indices $0 < \alpha_f \leq \beta_f < \infty$, and $E\left[|X_1|^{1+\epsilon}\right] < \infty$ for some $0 < \epsilon < \alpha_f$. Let $(N, C_1, C_2, \dots)$ be a nonnegative random vector independent of the $\{X_i\}$ with $N \in \mathbb{N} \cup \{\infty\}$ and satisfying $E\left[ \sum_{i=1}^N C_i^{\alpha_f - \epsilon} \right] < \infty$ and $E\left[ \sum_{i=1}^N C_i^{\beta_f + \epsilon} \right] < \infty$. If $E[N] < \infty$ then the condition  $E\left[ \sum_{i=1}^N C_i^{\alpha_f - \epsilon} \right] < \infty$ can be dropped. Let $Z_N = \sum_{i=1}^N C_i < \infty$ a.s. If any of the following holds,
\begin{enumerate} \renewcommand{\labelenumi}{\alph{enumi})}
\item $E[X_1] < 0$, or, 
\item $E[X_1] = 0$ and $P(Z_N > x) = O\left(\overline{F}(x)\right)$ as $x \to \infty$, or,
\item $E[X_1] > 0$ and $P(Z_N > x) = o\left(\overline{F}(x) \right)$ as $x \to \infty$,
\end{enumerate}
then, as $x \to \infty$,
\begin{equation} \label{eq:OnlySummands}
P\left( \sup_{1\leq k < N+1} \sum_{i=1}^k C_i X_i > x \right) \sim P\left( \sum_{i=1}^N C_i X_i > x \right) \sim E\left[ \sum_{i=1}^N \overline{F}(x/C_i) \right].
\end{equation}
\end{thm}

{\bf Remark:} It is known that when $N = n$ it is enough to have $E\left[ \sum_{i=1}^N C_i^{\beta_f+\epsilon} \right] < \infty$ for \eqref{eq:OnlySummands} to hold (see \cite{Wang_Tang_06, Zhang_Shen_Weng_09}). Note that for a finite number of terms this moment condition on the weights implies that
$$\left( E\left[Z_n^{\beta_f+\epsilon} \right] \right)^{1/(\beta_f+\epsilon)} \leq  \sum_{i=1}^n \left(E\left[C_i^{\beta_f+\epsilon}\right] \right)^{1/(\beta_f+\epsilon)} < \infty,$$
which in turn implies that $P(Z_n > x) = o\left(\overline{F}(x) \right)$ (since $x^{\beta_f+\epsilon} \overline{F}(x) \to \infty$). However, for $N = \infty$ and $\beta_f \geq 1$, the existing literature (e.g., \cite{Res_Will_91, Goovaerts_etal_05, Zhang_Shen_Weng_09}), which assumes $\overline{F} \in ER(-c,-d)$, requires the conditions $\sum_{i=1}^\infty \left( E\left[C_i^{d+\epsilon}\right] \right)^{1/(d+\epsilon)} < \infty$ and $\sum_{i=1}^\infty \left( E\left[C_i^{c-\epsilon}\right] \right)^{1/(d+\epsilon)} < \infty$, which again imply that $E\left[Z_\infty^{\beta_f+\epsilon}\right] < \infty$. In view of Theorem \ref{T.SummandsDominate}, the existing conditions are clearly too strong, and a simple example where \eqref{eq:OnlySummands} holds but $\sum_{i=1}^\infty \left( E\left[C_i^{d+\epsilon}\right] \right)^{1/(d+\epsilon)} = \infty$ is given below. Moreover, that the conditions of Theorem \ref{T.SummandsDominate} are close to being the best possible will follow from Theorem \ref{T.Both}. 

\begin{myexample}
Suppose that as $x \to \infty$, $\overline{F}(x) \asymp x^{-\alpha}$ for some $\alpha > 1$, $P(N> x) \asymp \overline{F}(x)$, and $E[X_1] = 0$. Furthermore, assume that the $\{C_i\}$ are i.i.d., independent of $N$, with $E[C_1^{\alpha+\epsilon}] < \infty$. Now write $\widetilde{C}_i = C_i \Indicator(i \leq N)$ so that $\sum_{i=1}^N C_i X_i = \sum_{i=1}^\infty \widetilde{C}_i X_i$, and note that for some constant $K > 0$, 
$$\sum_{i=1}^\infty \left( E\left[ \widetilde{C}_i^{\alpha+\epsilon} \right]  \right)^{1/(\alpha+\epsilon)} = \left( E[ C_1^{\alpha+\epsilon}] \right)^{1/(\alpha+\epsilon)} \sum_{i= 1}^\infty  P(N \geq i)^{1/(\alpha+\epsilon)} \geq  \sum_{i=1}^\infty \frac{K}{i^{\alpha/(\alpha+\epsilon)}} = \infty. $$
\end{myexample}

{\bf Remarks:} (i) The conditions of Theorem \ref{T.SummandsDominate} are very similar to those of Theorem 1 in \cite{Den_Foss_Kor_10} once we replace the random time $\tau$ by the random sum of the weights $Z_N = \sum_{i=1}^N C_i$. (ii) The stronger condition $E[|X_1|^{1+\epsilon}] < \infty$, instead of only $E[|X_1|] < \infty$, might be avoidable with a different proof technique.

The next result corresponds to the case where the behavior of the weighted random sum and its maximum might be influenced by both the one-big-jump principle and the distribution of the sum of the weights. This case also illustrates that when $E[X_1] > 0$, the conditions from Theorem \ref{T.SummandsDominate} are the best possible. 

\begin{thm} \label{T.Both}
Suppose $\{X_i\}$ is a sequence of i.i.d. random variables with right tail distribution $\overline{F} \in IR$, Matuszewska indices $0 < \alpha_f \leq \beta_f < \infty$, $E[X_1] > 0$, and $E\left[|X_1|^{1+\epsilon}\right] < \infty$ for some $0 < \epsilon < \alpha_f$. Let $(N, C_1, C_2, \dots)$ be a nonnegative random vector independent of the $\{X_i\}$ with $N \in \mathbb{N} \cup \{\infty\}$ and satisfying $E\left[ \sum_{i=1}^N C_i^{\alpha_f - \epsilon} \right] < \infty$ and $E\left[ \sum_{i=1}^N C_i^{\beta_f + \epsilon} \right] < \infty$. If $E[N] < \infty$ then the condition $E\left[ \sum_{i=1}^N C_i^{\alpha_f - \epsilon} \right] < \infty$ can be dropped.  Let $Z_N = \sum_{i=1}^N C_i < \infty$ a.s. and suppose further that its tail distribution $\overline{G} \in IR$. Then, as $x \to \infty$, 
$$P\left( \sup_{1\leq k < N+1} \sum_{i=1}^k C_i X_i > x \right) \sim P\left( \sum_{i=1}^N C_i X_i > x \right) \sim E\left[ \sum_{i=1}^N \overline{F}(x/C_i) \right] + P\left( \sum_{i=1}^N C_i > x/E[X_1] \right).$$
\end{thm}

\bigskip

{\bf Remark:} If $\{X_i\}$ is a sequence of i.i.d. random variables from $\mathcal{R}_{-\alpha}$ with $\alpha > 1$, then $E\left[ \sum_{i=1}^N \overline{F}(x/C_i) \right]$ can be replaced with $E\left[ \sum_{i=1}^N C_i^\alpha \right] \overline{F}(x)$ in Theorems \ref{T.SummandsDominate} and \ref{T.Both}. In this setting, Theorems \ref{T.SummandsDominate} and \ref{T.Both} are generalizations of Breiman's Theorem to more than one summand and dependent weights. 

The third, and the last, result corresponds to the case where the behavior of the weighted random sum  is dominated solely by the sum of the weights. Note that it is not necessary for the $\{X_i\}$ to have any particular structure beyond certain moments and the condition $P(X_1 > x) = o(P(Z_N > x))$ as $x \to \infty$. 

\begin{thm} \label{T.ZDominates}
Let $(N, C_1, C_2, \dots)$ be a nonnegative random vector with $N \in \mathbb{N} \cup \{\infty\}$. Define $Z_N = \sum_{i=1}^N C_i < \infty$ a.s. and assume that it has a right tail distribution $\overline{G} \in IR$ with Matuszewska indices $0 < \alpha_g \leq \beta_g < \infty$. Suppose $\{X_i\}$ is a sequence of i.i.d. random variables, independent of $(N, C_1, C_2, \dots)$, with $E[X_1] > 0$, and $E\left[|X_1|^{1+\epsilon}\right] < \infty$ for some $0< \epsilon < \alpha_g$. Suppose further that $E\left[ \sum_{i=1}^N C_i^{\alpha_g -\epsilon} \right] < \infty$, $E\left[ \sum_{i=1}^N C_i^{\beta_g+\epsilon} \right] < \infty$, and $P(X_1 > x) = o\left( P(Z_N > x) \right)$. If $E[N] < \infty$ then the condition $E\left[ \sum_{i=1}^N C_i^{\alpha_g - \epsilon} \right] < \infty$ can be dropped. Then, as $x \to \infty$, 
$$P\left( \sup_{1\leq k < N+1} \sum_{i=1}^k C_i X_i > x \right) \sim P\left( \sum_{i=1}^N C_i X_i > x \right) \sim P\left( \sum_{i=1}^N C_i > x/E[X_1] \right).$$
\end{thm}

\section{The upper bound} \label{S.UpperBound}

Before proceeding with the derivation of the auxiliary results that will be needed for the proofs of the main theorems, we state here the notation that will be used in the remainder of the paper, as well as the main assumption satisfied by the random variables $\{X_i\}$ and the vector $(N, C_1, C_2, \dots)$.

\begin{assumption} \label{A.Main}
Let $\{ X_i \}$ be a sequence of i.i.d. real-valued random variables with common tail distribution $\overline{F}(x) = P(X_1 > x)$ and finite mean $\mu = E[X_1]$, and let $(N, C_1, C_2, \dots)$ represent a nonnegative random vector, independent of the $\{X_i\}$, with $N \in \mathbb{N} \cup \{\infty\}$. The vector $(N, C_1, C_2, \dots)$ is assumed to be generally dependent and the weights $\{C_i\}$ are not necessarily identically distributed. 
\end{assumption}

We will also use $|| \cdot ||_p = (E[ | \cdot |^p ])^{1/p}$ to denote the $L_p-$norm, the operator $\#A$ to denote the cardinality of a set $A$, and the symbols $x \vee y = \max\{x, y\}$, $x \wedge y = \min\{x, y\}$.  The letter $K$ will be used to denote a generic positive constant, which is not always the same in different parts of the paper, i.e. $K = K+1$, $K = 2K$, etc. 

The following random variables will be used throughout the paper:
\begin{align*}
&S_k = \sum_{i=1}^k C_i X_i, \quad k \in \mathbb{N} \cup \{\infty\}, \\
&M_N = \sup_{1\leq k < N+1} S_k, \\
&Z_N = \sum_{i=1}^N C_i, \\
&I_N(t) = \#\{1 \leq i < N+1: C_i > t\}, \\
&J_N(t) = \#\{ 1 \leq i < N+1: C_i X_i > t\}, \\
&L_N(t) = \#\{1 \leq i < N+1: C_i X_i < -t\}.
\end{align*}
Note that when $N$ is finite a.s. the supremum in the definition of $M_N$ can be replaced by a maximum and all the ranges $1 \leq i < N+1$ can be replaced by $1 \leq i \leq N$. Recall that since the weights $\{C_i\}$ are nonnegative, the convention is that $\overline{F}(t/C_i) = 0$ and $F(-t/C_i) = 0$ for any $t \geq 0$ if $C_i = 0$. 

The first result in this section provides a bound for the partial maximum of sums of independent random variables with finite exponential moments.

\begin{lem} \label{L.ExpBound}
Let $\{V_i\}_{i \geq 1}$ be a sequence of independent random variables. Then, for all $\theta > 0$, 
$$P\left( \max_{1\leq k \leq m} \sum_{i=1}^k V_i > t \right) \leq e^{-\theta t} \prod_{i=1}^m \max\left\{1, E\left[ e^{\theta V_i } \right]  \right\}.$$
\end{lem}

\begin{proof}
The inequality trivially holds in case $E\left[ e^{\theta V_i} \right] = \infty$ for some $i$. Thus, we assume that $E\left[ e^{\theta V_i} \right] < \infty$ for all $i = 1, \dots, m$.  Let 
$$L_k = e^{\theta \sum_{i=1}^k V_i - \varphi_k(\theta) }, \qquad \varphi_k(\theta) = \log \prod_{i=1}^k E\left[ e^{\theta V_i}  \right].$$
Then $L_k$ is a nonnegative martingale satisfying $E[ L_k ] = 1$. Define $\tau = \inf\{ k \geq 1: \sum_{i=1}^k V_i > t\}$. Then, by Proposition 3.1 and Theorem 3.2 in Chapter XIII of \cite{Asm2003} there exists a probability measure $\widetilde{P}$ such that
\begin{align*}
P\left(   \max_{1\leq k \leq m} \sum_{i=1}^k V_i > t \right) &= P( \tau \leq m)   \\
&= \widetilde{E} \left[ L_{\tau}^{-1} \Indicator \left(  \tau \leq m \right)  \right] \\
&\leq  \widetilde{E} \left[ e^{-\theta \sum_{i=1}^\tau V_i } \Indicator \left(  \tau \leq m  \right)  \right]  \prod_{i=1}^m \max\left\{1, E\left[ e^{\theta V_i } \right]  \right\} \\
&\leq e^{-\theta t} \prod_{i=1}^{m} \max\left\{1, E\left[ e^{\theta V_i } \right]  \right\}.
\end{align*}
\end{proof}

The following result gives exponential bounds for sums of independent truncated random variables, and it follows the same classical heavy-tailed techniques from \cite{Nag82} and \cite{Bor00} (see also \cite{Borov_Borov_2008}). Note that all of the results in this and the next section are given for random variables satisfying only moment conditions, that is, neither the $\{X_i\}$ nor the vector $(N, C_1, C_2, \dots)$ are assumed to belong to any particular class of distributions.

\begin{lem} \label{L.TruncBound}
Suppose the $\{X_i\}$ and the vector $(N, C_1, C_2, \dots)$ satisfy Assumption \ref{A.Main} with  $\gamma_{\eta} = || X_1 ||_{\eta} < \infty$ for some $\eta > 1$. Then, for any $0 < u < v$ such that
$$\frac{1}{v} \leq \theta \triangleq \frac{(\eta-1)}{v} \log \left( \frac{v}{u}  \right) \leq \frac{1}{u}, $$
any $z > 0$, and any $A \subseteq \mathbb{R}$, we have
\begin{align*}
&P\left( \sup_{1\leq k < N+1} \sum_{i=1}^k C_i X_i \Indicator(C_i X_i \leq v) > z , \, Z_N  \in A, \, I_N(u/\gamma_{\eta}) = 0 \right) \leq E\left[ \Indicator\left( Z_N \in A \right) e^{-\theta  z + \left( \mu + \frac{K \gamma_{\eta}}{\log (v/u)} \right)^+ \theta Z_N  }  \right]  ,
\end{align*}
where $K = K(\eta) > 1$ is a constant that does not depend on the distributions of $X_1, N, C_1, C_2, \dots$. 
\end{lem}

\begin{proof}
Let $X \stackrel{\mathcal{D}}{=} X_1$,  $Y_i = C_i X_i$ and $S_k^{(v)} = Y_1\Indicator(Y_1 \leq v) + \dots + Y_k \Indicator(Y_k \leq v)$.  By conditioning on $(N, C_1, C_2, \dots)$ we obtain 
\begin{align*}
&P\left( \max_{1\leq k \leq N\wedge n} S_k^{(v)} > z , \, Z_N \in A, \, I_N( u/\gamma_{\eta}) = 0 \right) \\
&= E\left[ \Indicator\left( Z_N \in A, \, I_N(u/\gamma_{\eta}) = 0 \right) P\left( \left.  \max_{1\leq k \leq N\wedge n} S_k^{(v)} >  z  \right| N \wedge n, C_1, \dots, C_{N\wedge n} \right)  \right] .
\end{align*}
Note that conditional on $(N \wedge n, C_1, \dots, C_{N\wedge n})$, $S_k^{(v)}$ is a sum of independent random variables, so by Lemma~\ref{L.ExpBound},
\begin{align*}
P\left( \left.  \max_{1\leq k \leq N\wedge n} S_k^{(v)} >  z  \right| N \wedge n, C_1, \dots, C_{N\wedge n} \right)  &\leq e^{-\theta z} \prod_{i=1}^{N\wedge n} \max\left\{1, E\left[ \left. e^{\theta Y_i \Indicator(Y_i \leq v)} \right| N, C_1, \dots, C_{N\wedge n} \right] \right\} \\
&=  e^{-\theta z} \prod_{i=1}^{N\wedge n} \max\left\{ 1, E\left[ \left. e^{\theta Y_i \Indicator(Y_i \leq v)} \right| C_i  \right] \right\}.
\end{align*}

We now bound the individual expectations using integration by parts as follows,
\begin{align*}
E\left[ \left. e^{\theta Y_i \Indicator(Y_i \leq v)} \right| C_i \right] &= E\left[ \left. e^{\theta Y_i} \Indicator(Y_i \leq v) \right| C_i \right] + E\left[ \left.  \Indicator(Y_i > v) \right| C_i \right] \\
&= \int_{-\infty}^v e^{\theta t} P(Y_i \in dt | C_i) + P(Y_i > v | C_i) \\
&= P(Y_i \leq 1/\theta|C_i) +\int_{-\infty}^{1/\theta} \theta t P(Y_i \in dt | C_i) + e P(Y_i > 1/\theta|C_i) -e^{\theta v} P(Y_i > v | C_i)   \\
&\hspace{5mm} + P(Y_i > v | C_i) +  \int_{-\infty}^{1/\theta} (e^{\theta t}-1-\theta t)  P(Y_i \in dt | C_i)  + \theta \int_{1/\theta}^v e^{\theta t} P(Y_i > t| C_i) dt \\
&\leq 1 + \theta E[Y_i|C_i] + eP(Y_i > 1/\theta|C_i) +   \int_{-\infty}^{1/\theta} (e^{\theta t}-1-\theta t)  P(Y_i \in dt | C_i) \\
&\hspace{5mm} + \theta \int_{1/\theta}^v e^{\theta t} P(Y_i > t| C_i) dt. 
\end{align*}
If $\eta \geq 2$ then the inequality $e^t - 1 - t\leq t^2 e^{t^+}$ for $t \in \mathbb{R}$ gives
$$\int_{-\infty}^{1/\theta} (e^{\theta t}-1-\theta t)  P(Y_i \in dt | C_i)  \leq e\theta^2 \int_{-\infty}^{\infty} t^2 P(Y_i \in dt |C_i) = e\theta^2 E[Y_i^2 | C_i].$$
If $1 < \eta < 2$, then integration by parts, a change of variables, Markov's inequality, and the same inequality used above give
\begin{align*}
\int_{-\infty}^{1/\theta} (e^{\theta t}-1-\theta t)  P(Y_i \in dt | C_i)  &= \theta \int_0^{\infty} (1-e^{-\theta u}) P(Y_i \leq -u|C_i) du + \int_0^{1/\theta} (e^{\theta t} - 1 - \theta t) P(Y_i \in dt | C_i) \\
&\leq \theta \int_0^\infty (1- e^{-\theta u}) E[|Y_i|^{\eta} | C_i] u^{-\eta} du + e\theta^2 \int_0^{1/\theta} t^2 P(Y_i \in dt|C_i) \\
&\leq E[|Y_i|^{\eta} |C_i] \left( \theta^2 \int_0^{1/\theta} \frac{1-e^{-\theta u}}{\theta u} \cdot u^{1-\eta} du + \theta \int_{1/\theta}^\infty u^{-\eta} du \right) \\
&\hspace{5mm} + e\theta^{\eta} \int_0^{1/\theta} t^{\eta} P(Y_i \in dt | C_i)  \\
&\leq E[|Y_i|^{\eta} | C_i ] \left( \theta^2 \int_0^{1/\theta} u^{1-\eta} du + \frac{\theta^{\eta}}{\eta-1}   + e\theta^{\eta} \right) \\
&= \theta^{\eta} E[|Y_i|^{\eta} |C_i] \left( \frac{1}{2-\eta} + \frac{1}{\eta-1} + e  \right),
\end{align*}
where in the third inequality we used the observation that $1 - e^{-t} \leq t$ for all $t \geq 0$. We then have that
\begin{equation} \label{eq:firstmoments}
\int_{-\infty}^{1/\theta} (e^{\theta t}-1-\theta t)  P(Y_i \in dt | C_i)  \leq K_1 \theta^{\eta\wedge 2} C_i^{\eta\wedge 2} E[|X|^{\eta\wedge2}],
\end{equation}
where $K_1 = K_1(\eta) = e + ((2-\eta)^{-1}+(\eta-1)^{-1}) \Indicator(1 < \eta < 2)$. Also, for $\eta > 0$ we use Markov's inequality to obtain
\begin{align*}
e P(Y_i > 1/\theta|C_i) + \theta \int_{1/\theta}^v e^{\theta t} P(Y_i > t| C_i) dt &\leq e E[|Y_i|^{\eta} | C_i] \theta^{\eta} + E[|Y_i|^{\eta} | C_i] \int_{1/\theta}^v \theta e^{\theta t} t^{-\eta} dt.
\end{align*}
To analyze the remaining integral we split it as follows, 
\begin{align*}
\theta \int_{1/\theta}^v e^{\theta t} t^{-\eta} dt &\leq \theta^{1+\eta} \int_{1/\theta}^{(1/\theta) \vee (v/2)} e^{\theta t} dt + \theta \int_{v/2}^v e^{\theta t} t^{-\eta} \, dt \\
&\leq \theta^{\eta} e^{\theta v/2} + \theta v^{1-\eta} \int_{1/2}^1 e^{\theta v u} u^{-\eta} \, du \\
&\leq \theta^{\eta} e^{\theta v/2} + \theta v^{1-\eta} 2^{\eta} \int_{1/2}^1 e^{\theta v u} \, du \\
&\leq \theta^{\eta} e^{\theta v/2} +  2^{\eta} e^{\theta v}  v^{-\eta}.
\end{align*}
Hence, 
\begin{equation} \label{eq:exptail}
eP(Y_i > 1/\theta|C_i) + \theta \int_{1/\theta}^v e^{\theta t} P(Y_i > t | C_i) dt \leq K_2 e^{\theta v} v^{-\eta} C_i^{\eta} E[| X|^{\eta}],
\end{equation}
where $K_2 = K_2(\eta) = \sup_{t \geq 1} \left( et^{\eta} e^{-t} + t^{\eta} e^{-t/2} + 2^{\eta} \right)$. Combining \eqref{eq:firstmoments} and \eqref{eq:exptail} we obtain 
\begin{align*}
&E\left[ \left. e^{\theta Y_i \Indicator(Y_i \leq v)} \right| C_i \right] \Indicator(C_i \leq  u/\gamma_{\eta}) \\
&\leq \left( 1 + \theta C_i E[X] + K_1 \theta^{\eta\wedge 2} C_i^{\eta\wedge 2} E[| X|^{\eta\wedge2}] + K_2 e^{\theta v} v^{-\eta} C_i^{\eta} E[|X|^{\eta}] \right) \Indicator(C_i \leq u/\gamma_{\eta}) \\
&\leq 1 + \theta C_i E[X] + C_i \left( K_1 \theta^{\eta\wedge2} (u/\gamma_{\eta})^{\eta\wedge2-1} E[|X|^{\eta\wedge2}] + K_2 e^{\theta v} v^{-\eta} (u/\gamma_{\eta})^{\eta-1} E[| X|^{\eta}]  \right) .
\end{align*} 
We now use the observation that $E[|X|^{\eta\wedge2}] = ||X||_{\eta\wedge2}^{\eta\wedge2} \leq || X||_{\eta}^{\eta\wedge2} = \gamma_{\eta}^{\eta\wedge 2}$ to obtain
\begin{align*}
E\left[ \left. e^{\theta Y_i \Indicator(Y_i \leq v)} \right| C_i \right] \Indicator(C_i \leq u/\gamma_{\eta}) &\leq 1 + \theta C_i \mu + C_i \gamma_{\eta} \left( K_1 \theta^{\eta\wedge 2} u^{\eta\wedge2-1} + K_2 e^{\theta v} v^{-\eta}  u^{\eta-1} \right) \\
&\leq 1 + \theta C_i \mu + \theta C_i \gamma_{\eta} K_3 \left( \theta^{\eta\wedge2-1} u^{\eta\wedge2-1} + e^{\theta v} v^{-\eta} \theta^{-1} u^{\eta-1} \right) \\
&\triangleq 1 + \theta C_i \mu +  \theta C_i \gamma_\eta a(\theta,u,v) ,
\end{align*} 
where $K_3 = K_3(\eta) = \max\{ K_1(\eta), K_2(\eta)\}$.  By using the inequality $1 + t \leq e^t$ for all $t \in \mathbb{R}$, it follows that
\begin{align}
&P\left( \max_{1\leq k\leq N\wedge n} S_k^{(v)} > z, \, Z_N \in A, \, I_N(u/\gamma_{\eta}) = 0 \right) \notag \\
&\leq E\left[ \Indicator \left( Z_N \in A \right) e^{-\theta z} \prod_{i=1}^{N\wedge n} \max\left\{ 1,  1 + \theta \mu C_i+ \theta a(\theta, u,v) \gamma_{\eta} C_i  \right\} \right] \notag \\
&\leq E\left[ \Indicator \left(  Z_N \in A \right) e^{-\theta z} \prod_{i=1}^{N\wedge n} \max\left\{1, e^{ \theta \mu C_i + \theta a(\theta,u,v) \gamma_{\eta} C_i} \right\} \right] \notag \\
&= E\left[ \Indicator\left(  Z_N \in A \right) e^{-\theta z +  \left(\mu + a(\theta,u,v)\gamma_\eta \right)^+ \theta Z_{N\wedge n}}  \right]  . \notag
\end{align}
Now choose 
$$\theta = \frac{1}{v} \log \left( (v/u)^{\eta-1} \right),$$
which by assumption satisfies $1/v \leq \theta \leq 1/u$, and note that
\begin{align*}
a(\theta, u,v) &= \frac{K_3}{(\eta-1)\log(v/u)} \left(1 + \frac{((\eta-1)\log(v/u))^{\eta\wedge2}}{(v/u)^{\eta\wedge2-1}} \right) \\
&\leq \frac{K_3}{(\eta-1) \log (v/u) } \left(1 + (\eta-1)^{\eta \wedge 2} \sup_{t \geq 1}  \frac{(\log t)^{\eta \wedge 2}}{t^{\eta\wedge 2-1}} \right).
\end{align*}
Defining $K = K(\eta) = \frac{K_3}{\eta-1} \left(1 + (\eta-1)^{\eta \wedge 2} \sup_{t \geq 1}  \frac{(\log t)^{\eta \wedge 2}}{t^{\eta\wedge 2-1}} \right)$ gives
\begin{align*}
&P\left( \max_{1\leq k\leq N\wedge n} S_k^{(v)} > z, \, Z_N \in A, \, I_N(u/\gamma_{\eta}) = 0 \right) \\
&\leq  E\left[ \Indicator\left( Z_N \in A \right) e^{-\theta z +  \left(\mu + \frac{K \gamma_{\eta}}{\log(v/u)} \right)^+ \theta Z_{N\wedge n} }  \right] .
\end{align*}
The result now follows by taking $n \to \infty$.  
\end{proof}

\bigskip

The main result of this section, given in Proposition \ref{P.UpperBounds}, provides upper bounds for $P(M_N > x)$. The idea of the proof is to split this probability into several smaller probabilities corresponding to the different possible behaviors of $Z_N$ and $J_N(\cdot)$. The bound derived in Lemma \ref{L.TruncBound} will be essential to the analysis of all the probabilities involving truncated summands. The lemma given below provides a bound for the probability of two or more summands being large. 

\begin{lem} \label{L.SecondMax}
Suppose the $\{X_i\}$ and the vector $(N, C_1, C_2, \dots)$ satisfy Assumption \ref{A.Main} with $\gamma_{1+\epsilon} = ||X_1||_{1+\epsilon} < \infty$ for some $\epsilon > 0$.  Let $0 < \nu < 1$, $w = x^{1-\nu}/\gamma_{1+\epsilon}$ and $y = x/\log x$. Fix $c > 0$. Then, there exist constants $K, x_0 > 0$ such that for all $x \geq x_0$, 
$$P\left( J_{N}(y) \geq 2, \, Z_N \leq cx, \, I_N(w) = 0 \right) \leq \frac{K (\log x)^{1+\epsilon} }{ x^{\epsilon\nu} } E\left[  \Indicator\left( I_N(w) = 0 \right) \sum_{i = 1}^N \overline{F}(y/ C_i)  \right] .$$
\end{lem}

\begin{proof}
We start by conditioning on $\mathcal{F} = \sigma(N, C_1, C_2, \dots)$ to obtain
\begin{align*}
&P\left( J_N(y) \geq 2, \, Z_N \leq cx, \, I_N(w) = 0 \right) \\
&= E\left[ \Indicator \left( Z_N \leq cx, \, I_N(w) = 0 \right) E\left[ \left.  \Indicator\left( J_N(y) \geq 2 \right) \right| \mathcal{F} \right] \right] \\
&= E\left[ \Indicator \left( Z_N \leq cx, \, I_N(w) = 0 \right) E\left[ \left.  \Indicator \left( \bigcup_{1\leq i < j < N+1} \left\{ C_i X_i > y, C_j X_j > y \right\} \right) \right| \mathcal{F} \right] \right] \\
&\leq E\left[ \Indicator \left( Z_N \leq cx, \, I_N(w) = 0 \right) \sum_{1 \leq i < j < N+1} E \left[  \left. \Indicator \left( C_i X_i > y, C_j X_j > y\right) \right| \mathcal{F} \right] \right]  \\
&= E\left[ \Indicator \left( Z_N \leq cx, \, I_N(w) = 0 \right) \sum_{1 \leq i < j < N+1} \overline{F}(y/ C_i) \overline{F}(y/C_j) \right] \\
&\leq E\left[ \Indicator \left( Z_N \leq cx, \, I_N(w) = 0 \right) \left( \sum_{i = 1}^N \overline{F}(y/ C_i) \right)^2 \right] ,  
\end{align*}
where in the third equality we used the conditional independence of the $\{C_i X_i\}$ given $\mathcal{F}$ and the independence of the $\{X_i\}$ and $(N, C_1, C_2, \dots)$. We now use Markov's inequality to obtain
\begin{align*}
\sum_{j=1}^N \overline{F}(y/C_j) &\leq y^{-1-\epsilon} E[|X_1|^{1+\epsilon}] \sum_{j=1}^N C_j^{1+\epsilon} \leq K y^{-1-\epsilon}  \left( \sup_{1\leq j <N+1} C_j \right)^{\epsilon}  Z_N.
\end{align*}
It follows that
\begin{align*}
E\left[ \Indicator \left( Z_N \leq cx, \, I_N(w) = 0 \right) \left( \sum_{i = 1}^N \overline{F}(y/ C_i ) \right)^2 \right] &\leq  K y^{-1-\epsilon} w^{\epsilon} x E\left[  \Indicator \left( I_N(w) = 0 \right) \sum_{i = 1}^N \overline{F}(y/ C_i )  \right] \\
&\leq \frac{K (\log x)^{1+\epsilon} }{ x^{\epsilon\nu} } E\left[  \Indicator \left( I_N(w) = 0 \right) \sum_{i = 1}^N \overline{F}(y/ C_i )  \right]  .
\end{align*}
\end{proof}

The next preliminary lemma shows that if the summands are heavily truncated, then the supremum of the sums is unlikely to be large.  

\begin{lem} \label{L.StrongTrunc}
Suppose the $\{X_i\}$ and the vector $(N, C_1, C_2, \dots)$ satisfy Assumption \ref{A.Main} with $\gamma_{1+\epsilon} = ||X_1||_{1+\epsilon} < \infty$ for some $\epsilon > 0$. Let $0 < \nu < 1$, $w = x^{1-\nu}/\gamma_{1+\epsilon}$, $y = x/\log x$ and $0 < 1/\sqrt{\log x} \leq \delta < 1$. Then,    
$$P\left( \sup_{1\leq k < N+1}  \sum_{i=1}^{k} C_i X_i \Indicator(C_iX_i\leq y) > \delta x,  \, Z_N \leq y , \, I_N(w) = 0 \right) = o\left(x^{-h} \right)$$
as $x \to \infty$, for any $h > 0$.  
\end{lem}

\begin{proof}
We use Lemma \ref{L.TruncBound} with $A = (-\infty, y]$, $v = y$, $z = \delta x$, and $u = x^{1-\nu}$.  Then
\begin{align*}
&P\left( \sup_{1\leq k <  N+1}  \sum_{i=1}^{k} C_i X_i \Indicator(C_iX_i\leq y) > \delta x, \, Z_N \leq y, \, I_N(w) = 0  \right) \\
&\leq E\left[ \Indicator(Z_N \leq y) e^{-\theta  \delta x + \left( \mu + \frac{K\gamma_{1+\epsilon}}{\log(y/u)} \right)^+ \theta Z_N  } \right] ,
\end{align*}
where $\theta = \frac{\epsilon}{y} \log(y/u) = \frac{\epsilon\log x}{x}  \log\left( \frac{x^\nu}{\log x} \right)$. Note that for this choice of $y$ and $u$ there exists $x_0 = x_0(\beta,\epsilon) > 0$ such that the conditions on $\theta$ required by Lemma \ref{L.TruncBound} are satisfied for all $x \geq x_0$.  Moreover, on the set $\{Z_N \leq y\}$ we have
\begin{align*}
-\theta \delta x + \left( \mu +  \frac{K\gamma_{1+\epsilon}}{\log(y/u)} \right)^+ \theta Z_N &\leq -\theta \delta x \left(1 - \frac{|\mu|y  \sqrt{\log x}}{x} - \frac{K \gamma_{1+\epsilon} y  \sqrt{\log x}}{x\log(x^\nu/\log x)} \right) \\
&\leq -\theta \delta x \left( 1 - \frac{|\mu|}{\sqrt{\log x}} - \frac{K\gamma_{1+\epsilon} }{\sqrt{\log x}} \right) \\
&\leq - \epsilon \nu \delta (\log x)^2 \left(1 - \frac{\log\log x}{\nu \log x} \right) \left( 1 - \frac{2K\gamma_{1+\epsilon}}{\sqrt{\log x}}  \right),
\end{align*}
where in the last inequality we used $|\mu| \leq E[|X_1|] \leq (E[|X_1|^{1+\epsilon}] )^{1/(1+\epsilon)} = \gamma_{1+\epsilon}$ and $K > 1$. We then have, for sufficiently large $x$, 
\begin{equation} \label{eq:TruncExpFinal}
E\left[ \Indicator(Z_N \leq y) e^{-\theta \delta x + \left( \mu + \frac{K\gamma_{1+\epsilon}}{\log(y/w)} \right)^+ \theta Z_N  } \right]  \leq e^{-\epsilon \nu \delta (\log x)^2 \varphi(x)},
\end{equation}
where $\varphi(x) = \left(1 - \frac{\log\log x}{\nu \log x} \right) \left( 1 - \frac{2K\gamma_{1+\epsilon}}{\sqrt{\log x}}  \right)$.  Since $\delta > 1/\sqrt{\log x}$, it holds that
$$e^{-\epsilon \nu\delta(\log x)^2\varphi(x)} \leq e^{-\epsilon \nu(\log x)^{3/2} \varphi(x)} = o\left( x^{h} \right)$$
as $x \to \infty$ for any $h > 0$. 
\end{proof}

\bigskip

The last preliminary lemma of this section provides a bound for the case when the summands are moderately truncated and $Z_N$ is not too large. 

\begin{lem} \label{L.InterTrunc}
Suppose the $\{X_i\}$ and the vector $(N, C_1, C_2, \dots)$ satisfy Assumption \ref{A.Main} with $\mu \geq 0$ and $\gamma_{1+\epsilon} = ||X_1||_{1+\epsilon} < \infty$ for some $\epsilon > 0$.  Let $0 < \nu < 1$, $w = x^{1-\nu}/\gamma_{1+\epsilon}$, $y = x/\log x$ and $0 < 1/\sqrt{\log x} \leq \delta <1$. Then,  as $x \to \infty$, 
\begin{align*}
&P\left( M_N > x, \,  J_N((1-\delta)x) = 0,  \, y < Z_N \leq x/(\mu+\delta), \, I_N(w) = 0  \right) \\
&= O\left( x^{-\epsilon\nu/2} P(Z_N > y) + e^{-\frac{\epsilon\nu\sqrt{\log x}}{\mu}} P(Z_N > x/(2\mu)) \Indicator(\mu > 0) \right) .
\end{align*}
\end{lem}

\begin{proof}
Let $A = (y, x/(\mu+\delta))$, $v = (1-\delta)x$ and $u = x^{1-\nu}$. Then, by Lemma \ref{L.TruncBound}, we have 
\begin{align}
&P\left( M_N > x, \, J_N((1-\delta)x) = 0,  \, y < Z_N \leq x/(\mu+\delta) , \, I_N(w) = 0 \right) \notag \\
&\leq P\left( \sup_{1 \leq k < N+1} \sum_{i=1}^k C_i X_i \Indicator(C_i X_i \leq (1-\delta)x) > x, \, y < Z_N \leq x/(\mu+\delta) , \, I_N(w) = 0  \right) \notag \\
&\leq E\left[ \Indicator(y < Z_N \leq x/(\mu+\delta)) e^{-\theta x + \left( \mu + \frac{K\gamma_{1+\epsilon}}{\log((1-\delta)x^\nu)} \right)^+ \theta Z_N  } \right] , \label{eq:MGF_Z}
\end{align}
where $\theta = \frac{\epsilon}{(1-\delta)x} \log ((1-\delta) x^\nu)$. We now separate the rest of the analysis into two cases.

{\bf Case 1:} $\mu = 0$.

We have that \eqref{eq:MGF_Z} is bounded by
$$e^{-\theta x \left(1 - \frac{K\gamma_{1+\epsilon} }{\delta \log((1-\delta) x^\nu)} \right)} P(Z_N > y) \leq K e^{-\epsilon\nu \left(1 - \frac{K\gamma_{1+\epsilon} \sqrt{\log x} }{\log((1-\delta) x^\nu)} \right) \log x} P(Z_N > y) \leq \frac{K}{x^{\epsilon\nu/2}}  P(Z_N > y),$$
for sufficiently large $x$. 

{\bf Case 2:} $\mu > 0$.

Note that $\frac{K \gamma_{1+\epsilon}}{\log((1-\delta)x^\nu)} \theta Z_N \leq \frac{\epsilon K \gamma_{1+\epsilon}}{(1-\delta)(\mu+\delta)} < \infty$, so \eqref{eq:MGF_Z} is bounded by
\begin{align*}
K E\left[ \Indicator(y < Z_N \leq x/(\mu+\delta)) e^{-\theta \left( x - \mu Z_N \right) } \right] &\leq K  E\left[ \Indicator(y < Z_N \leq x/(\mu+\delta)) e^{-\frac{\epsilon \nu\log x}{x} \left( x - \mu Z_N \right) } \right] \\
&= \frac{K}{x^{\epsilon\nu}} E\left[ \Indicator(y < Z_N \leq x/(\mu+\delta)) e^{\frac{\epsilon\nu \mu \log x}{x} Z_N  } \right] .
\end{align*}
Now note that by writing $\Indicator(y < Z_N \leq x/(\mu+\delta)) \leq \Indicator(y < Z_N \leq x/(2\mu))+ \Indicator(x/(2\mu) < Z_N \leq x/(\mu+\delta))$ (if $\delta \geq \mu$ the second indicator is zero) we obtain
\begin{align*}
\frac{1}{x^{\epsilon\nu}} E\left[ \Indicator(y < Z_N \leq x/(\mu+\delta)) e^{\frac{\epsilon\nu \mu \log x}{x} Z_N  } \right] &\leq \frac{1}{x^{\epsilon\nu/2}} P(Z_N > y) + \frac{e^{\frac{\epsilon\nu \mu \log x}{\mu+\delta}}}{x^{\epsilon\nu}} P(Z_N > x/(2\mu)). 
\end{align*}
Since $x^{-\epsilon\nu} e^{\frac{\epsilon\nu \mu \log x}{\mu+\delta}} = e^{-\frac{\epsilon\nu \delta \log x}{\mu+\delta}} \leq  \exp\left\{-\frac{\epsilon\nu \sqrt{\log x}}{\mu + 1/\sqrt{\log x}} \right\} \leq K e^{-\frac{\epsilon\nu\sqrt{\log x}}{\mu}}$, the result follows. 
\end{proof}

\bigskip

We are now ready to provide upper bounds for $P(M_N > x)$. As mentioned earlier, the idea is to split the probability into all the different combinations of events relating $Z_N$ and $J_N(\cdot)$. We emphasize again that no particular structure on the distributions of $Z_N$ or the $\{X_i\}$ is imposed beyond moment conditions. 

\begin{prop} \label{P.UpperBounds}
Suppose the $\{X_i\}$ and the vector $(N, C_1, C_2, \dots)$ satisfy Assumption \ref{A.Main} with $\gamma_{1+\epsilon} = ||X_1||_{1+\epsilon} < \infty$ for some $\epsilon > 0$. In addition, assume that $E\left[ \sum_{i=1}^N C_i^{\beta+\epsilon} \right] < \infty$ for some $\beta > 0$. Then, there exist constants $K, x_0 > 0$ such that for all $x\geq x_0$ and $0 < 1/\sqrt{\log x} \leq \delta < 1$, 
\begin{enumerate} \renewcommand{\labelenumi}{\alph{enumi})}
\item For $\mu \geq 0$, 
\begin{align*}
P(M_N > x) &\leq E\left[ \Indicator\left( I_N(w) = 0 \right) \sum_{i=1}^N \overline{F}((1-\delta)x/C_i )  \right] + P\left( (\mu+\delta) Z_N > x   \right)  \\
&\hspace{5mm} + K \left( \frac{(\log x)^{1+\epsilon} }{ x^{\epsilon\nu} } E\left[  \Indicator\left( I_N(w) = 0 \right) \sum_{i = 1}^N \overline{F}(y/C_i )  \right]  + \frac{1}{x^{\beta+\epsilon/2}}  \right) \\
&\hspace{5mm} + K \left(  x^{-\epsilon\nu/2} P(Z_N > y)  + e^{-\frac{\epsilon\nu\sqrt{\log x}}{\mu}} P(Z_N > x/(2\mu)) \Indicator(\mu > 0) \right),
\end{align*}
\item For $\mu < 0$, 
\begin{align*}
P(M_N > x) &\leq E\left[ \Indicator\left( I_N(w) = 0 \right) \sum_{i=1}^N \overline{F}((1-\delta)x/C_i )  \right] \\
&\hspace{5mm} +  K \left( \frac{(\log x)^{1+\epsilon} }{ x^{\epsilon\nu} } E\left[  \Indicator\left( I_N(w) = 0 \right) \sum_{i = 1}^N \overline{F}(y/C_i )  \right] + \frac{1}{x^{\beta+\epsilon/2}}  \right),
\end{align*}
\end{enumerate}
where $y = x/\log x$, $\nu = \epsilon/(2(\beta+\epsilon))$ and $w = x^{1-\nu}/\gamma_{1+\epsilon}$. 
\end{prop}

\begin{proof}
We separate the analysis into two cases, $\mu \geq 0$ and $\mu < 0$. 

{\bf Case 1:} $\mu \geq 0$.

We start by splitting the probability as follows:
\begin{align}
P\left( M_N > x \right) &\leq P\left( M_N > x, \, (\mu+\delta) Z_N \leq x \right)  + P\left( (\mu+\delta) Z_N > x  \right) \notag \\
&\leq P\left( M_N > x, \, J_N((1-\delta) x) = 0, \, (\mu+\delta) Z_N \leq x, \, I_N(w) = 0 \right) + P(I_N(w) \geq 1)  \label{eq:Negligible} \\
&\hspace{5mm} + P\left( J_N((1-\delta) x) \geq 1, \, I_N(w) = 0 \right)  + P\left( (\mu+\delta) Z_N > x   \right). \notag
\end{align}
Let $\mathcal{F} = \sigma(N, C_1, C_2, \dots)$ and recall that the $\{X_i\}$ are independent of $(N, C_1,C_2, \dots)$. Then, from the union bound we obtain,
\begin{align}
P\left( J_N((1-\delta) x) \geq 1, \, I_N(w) = 0 \right) &= E \left[ \Indicator\left( I_N(w) = 0 \right) E\left[ \left. \Indicator\left(\bigcup_{i=1}^N \{ C_i X_i > (1-\delta) x\}  \right) \right| \mathcal{F} \right] \right] \notag \\
&\leq E \left[ \Indicator\left( I_N(w) = 0 \right) \sum_{i=1}^N E\left[ \left. \Indicator\left(C_i X_i > (1-\delta) x \right) \right| \mathcal{F} \right] \right] \notag \\
&= E \left[ \Indicator\left( I_N(w) = 0 \right) \sum_{i=1}^N \overline{F}((1-\delta)x/C_i )  \right]. \label{eq:OneBigJump}
\end{align}

Applying the union bound, Fubini's Theorem, and the conditional Markov inequality, we obtain
\begin{align}
P\left(  I_N(w) \geq 1 \right) &= E\left[ \Indicator\left( \bigcup_{i=1}^N \{C_i > w\}  \right) \right] \leq E\left[ \sum_{i=1}^N \Indicator(C_i > w)  \right] \notag \\
&= \sum_{i=1}^\infty P\left(   C_i > w, \, N \geq i \right) = \sum_{i=1}^\infty E\left[ \Indicator(N \geq i) E[ \Indicator(C_i > w ) | N] \right] \notag  \\
&\leq \frac{1}{w^{\beta+\epsilon}} \sum_{i=1}^\infty E\left[  C_i^{\beta+\epsilon} \Indicator(N \geq i)  \right] =  \frac{\gamma_{1+\epsilon}^{\beta+\epsilon}}{x^{\beta+\epsilon/2} } E\left[ \sum_{i=1}^N C_i^{\beta+\epsilon}   \right] . \label{eq:CMax}
\end{align}
Now, to analyze the first probability in \eqref{eq:Negligible}, split it according to how many summands are greater than $y$ as follows:
\begin{align}
&P\left( M_N > x, \, J_N((1-\delta) x) = 0,  \, Z_N \leq y, \, I_N(w) = 0 \right) \notag \\
&\hspace{5mm} + P\left( M_N > x, \, J_N((1-\delta) x) = 0, \, y < Z_N \leq x/(\mu+\delta), \, I_N(w) = 0   \right)   \notag \\
&\leq P\left( M_N > x , \, J_N(y) = 0,  \, Z_N \leq y, \, I_N(w) = 0 \right) \label{eq:TruncSum1} \\
&\hspace{3mm} + P\left( M_N > x, \, J_N(y) = 1, \, J_N((1-\delta) x) = 0,   \, Z_N \leq y, \, I_N(w) = 0 \right) \label{eq:TruncSum2} \\
&\hspace{3mm} + P\left( M_N > x, \, J_N(y) \geq 2, \, J_N((1-\delta)x) = 0, \,  Z_N \leq y, \, I_N(w) = 0  \right) \label{eq:SecondMin} \\
&\hspace{3mm} + P\left( M_N > x, \,  J_N((1-\delta)x) = 0,  \, y < Z_N \leq x/(\mu+\delta), \, I_N(w) = 0  \right). \label{eq:Intermediate} 
\end{align}
We start by analyzing \eqref{eq:TruncSum2}, which we can bound by separating the summands in $M_N$ into those that are smaller than or equal to $y$ and those that are greater than $y$, as the following derivation shows,
\begin{align}
&P\left( M_N > x , \, J_N(y) = 1, \, J_N((1-\delta) x) = 0, \, Z_N \leq y, \, I_N(w) = 0  \right) \notag \\
&= P\left( \sup_{1\leq k < N+1} \left\{ \sum_{i=1}^{k} C_i X_i \Indicator(C_iX_i\leq y) + \sum_{i=1}^{k} C_i X_i \Indicator(C_iX_i >  y) \right\}   > x, \right. \notag \\
&\hspace{5mm} \left. \phantom{\left\{ \sum_i^k \right\}} J_N(y) = 1, \, J_N((1-\delta) x) = 0, \,  Z_N \leq y, \, I_N(w) = 0  \right) \notag  \\
&\leq P\left( \sup_{1\leq k < N+1}  \sum_{i=1}^{k} C_i X_i \Indicator(C_iX_i\leq y) > \delta x, \, Z_N \leq y, \, I_N(w) = 0  \right) . \label{eq:TruncLem2}
\end{align}
We can bound \eqref{eq:TruncSum1} similarly to obtain
\begin{align}
&P\left( M_N  > x, \, J_N(y) = 0, \, Z_N \leq y, \, I_N(w) = 0 \right) \notag \\
&= P\left( \sup_{1\leq k < N+1} \sum_{i=1}^k C_i X_i \Indicator(C_i X_i \leq y) > x , \, J_N(y) = 0, \, Z_N \leq  y, \, I_N(w) = 0 \right) \notag \\
&\leq P\left( \sup_{1\leq k < N+1} \sum_{i=1}^k C_i X_i \Indicator(C_i X_i \leq y) > x,  \, Z_N \leq y, \, I_N(w) = 0 \right). \label{eq:TruncLem1}
\end{align}
Clearly, \eqref{eq:TruncLem1} is not greater than \eqref{eq:TruncLem2}, and to bound \eqref{eq:TruncLem2} we use Lemma \ref{L.StrongTrunc} to obtain 
\begin{align*}
P\left( \sup_{1\leq k < N+1}  \sum_{i=1}^{k} C_i X_i \Indicator(C_iX_i\leq y) > \delta x,  \, Z_N \leq y, \, I_N(w) = 0  \right) &\leq  \frac{K}{x^{h}}
\end{align*}
for any $h > 0$ (in particular, $h = \beta+\epsilon/2$). 
 
By Lemma \ref{L.SecondMax} we have that \eqref{eq:SecondMin} is bounded by
$$P\left( J_{N}(y) \geq 2, \, Z_N \leq y, \, I_N(w) = 0 \right) \leq  \frac{K (\log x)^{1+\epsilon} }{ x^{\epsilon\nu} } E\left[  \Indicator\left( I_N(w) = 0 \right) \sum_{i = 1}^N \overline{F}(y/ C_i)  \right]  .$$

Finally, by Lemma \ref{L.InterTrunc}, \eqref{eq:Intermediate} is bounded by
\begin{align*}
&P\left( M_N > x, \,  J_N((1-\delta)x) = 0,  \, y < Z_N \leq x/(\mu+\delta), \, I_N(w) = 0  \right) \\
&\leq K  \left( x^{-\epsilon\nu/2} P(Z_N > y) + e^{-\frac{\epsilon\nu\sqrt{\log x}}{\mu}} P(Z_N > x/(2\mu)) \Indicator(\mu > 0) \right).
\end{align*}
This completes the case.

{\bf Case 2:} $\mu < 0$. 

The case of negative mean requires some additional work, since in order to use the preliminary lemmas we need to have some control over $Z_N$. For this purpose let $\kappa = 2(\beta+\epsilon)/(\nu\epsilon|\mu|)$ and define $\tau = \sup\{ 1 \leq n \leq N: Z_n \leq \kappa x \}$. Now split the probability of interest as follows:
\begin{align}
P\left( M_N > x \right) &\leq P\left( J_N((1-\delta) x) \geq 1, \, I_N(w) = 0  \right) + P\left( I_N(w) \geq 1 \right) \label{eq:OnesDone}   \\
&\hspace{5mm} + P\left( M_N > x, \,  J_N((1-\delta)x) = 0, \, I_N(w) = 0 \right), \notag
\end{align}
and note that the probabilities in \eqref{eq:OnesDone} are bounded by \eqref{eq:OneBigJump} and \eqref{eq:CMax}.  For the remaining probability we use the union bound to obtain
\begin{align}
&P(M_N > x, \, J_N((1-\delta)x) = 0, \, I_N(w) = 0) \notag \\
&\leq P\left( \sup_{1\leq k < \tau+1} S_k > x, \, J_N((1-\delta)x) = 0, \, I_N(w) = 0 \right) \label{eq:notLargeYet} \\
&\hspace{5mm} + P\left(  \sup_{\tau < k < N+1} S_k > x, \, J_N((1-\delta)x) = 0, \, I_N(w) = 0, \, \tau < N \right). \label{eq:veryLarge}
\end{align}
Since $\tau \leq N$ and $Z_\tau \leq \kappa x$, \eqref{eq:notLargeYet} is bounded by
$$P\left( M_\tau > x, \, J_\tau((1-\delta)x) = 0, \, Z_\tau \leq \kappa x, \, I_N(w) = 0 \right).$$
We now split this last probability in a similar way to the previous case:
\begin{align}
&P\left( M_\tau > x, \, J_\tau((1-\delta)x) = 0, \, Z_\tau \leq \kappa x, \, I_N(w) = 0 \right) \notag  \\
&\leq P\left( M_\tau > x, \, J_\tau(y) = 0, \, Z_\tau \leq \kappa x, \, I_\tau(w) = 0 \right) \label{eq:NegTrunc1} \\
&\hspace{5mm} + P\left( M_\tau > x, \, J_\tau(y) = 1, \, J_\tau((1-\delta)x) = 0, \, Z_\tau \leq \kappa x, \, I_\tau(w) = 0 \right) \label{eq:NegTrunc2} \\
&\hspace{5mm} + P\left( J_\tau(y) \geq 2, \, Z_\tau \leq \kappa x, \, I_N(w) = 0 \right) .\label{eq:NegSecondMin}
\end{align}

By using the same arguments from the case $\mu \geq 0$, we obtain that the sum of the probabilities in  \eqref{eq:NegTrunc1} and \eqref{eq:NegTrunc2} is bounded by
\begin{align*}
2 P\left( \sup_{1\leq k \leq \tau} \sum_{i=1}^k C_i X_i \Indicator(C_i X_i \leq y) > \delta x, \, Z_\tau \leq \kappa x, \, I_\tau(w) = 0 \right)  ,
\end{align*}
which by Lemma \ref{L.TruncBound} (with $u = x^{1-\nu}$, $v = y$, $A = (-\infty, \kappa x]$, and $N = \tau$) is in turn bounded by
$$2 E\left[ \Indicator(Z_\tau \leq \kappa x) e^{-\theta\delta x + \left( \mu +  \frac{K\gamma_{1+\epsilon}}{\log(x^\nu/\log x)} \right)^+ \theta Z_\tau } \right]  \leq 2 e^{-\theta \delta x}  , $$
for sufficiently large $x$ and $\theta = \frac{\epsilon \log x}{x} \log (x^\nu/\log x)$. We now note that since $\delta \geq 1/\sqrt{\log x}$, then $e^{-\delta \theta x} \leq e^{- \epsilon\sqrt{\log x} \log(x^\nu/\log x)} = o\left( x^{-\beta-\epsilon/2} \right)$ as $x \to \infty$. By adapting the proof of Lemma \ref{L.SecondMax} to substitute $N$ by $\tau$ but keeping the condition $I_N(w) = 0$, we obtain that \eqref{eq:NegSecondMin} is bounded by
$$\frac{K(\log x)^{1+\epsilon}}{x^{\epsilon\nu}} E\left[ \Indicator(I_N(w) = 0) \sum_{i=1}^\tau \overline{F}(y/C_i) \right] \leq \frac{K(\log x)^{1+\epsilon}}{x^{\epsilon\nu}} E\left[ \Indicator(I_N(w) = 0) \sum_{i=1}^N \overline{F}(y/C_i) \right] .$$

Finally, to analyze \eqref{eq:veryLarge} let $\widetilde{X}_i = X_i - \mu/2$, $\widetilde{S}_k = C_1 \widetilde{X}_1 + \dots + C_k  \widetilde{X}_k$, and note that we can write the probability as
\begin{align}
&P\left(  \sup_{\tau < k < N+1} ( \widetilde{S}_k  - |\mu| Z_k/2) > x, \, J_N((1-\delta)x) = 0, \, I_N(w) = 0, \, \tau < N \right) \notag \\
&\leq P\left(  \sup_{\tau < k < N+1} \widetilde{S}_k  - |\mu| Z_{\tau+1}/2 > x, \, J_N((1-\delta)x) = 0, \, I_N(w) = 0, \,\tau < N \right) \notag \\
&\leq P\left(  \sup_{\tau < k < N+1} \sum_{i=1}^k C_i \widetilde{X}_i \Indicator(C_i \widetilde{X}_i \leq (1-\delta)x + C_i |\mu|/2)  > (1 + |\mu| \kappa/2) x, \, I_N(w) = 0, \,\tau < N \right) \notag \\
&\leq P\left(  \sup_{1\leq k < N+1} \sum_{i=1}^k C_i \widetilde{X}_i \Indicator(C_i \widetilde{X}_i \leq (1-\delta)x + |\mu|w/2)  > (1 +|\mu| \kappa/2) x, \, I_N(w) = 0 \right). \label{eq:LastNegBd}
\end{align}
Applying Lemma \ref{L.TruncBound} (with $u = x^{1-\nu}$, $v = (1-\delta)x+|\mu|w/2$, and $A = \mathbb{R}$) gives that \eqref{eq:LastNegBd} is bounded by
$$E\left[ e^{-\phi (1+|\mu|\kappa/2) x + \left( \mu/2 + \frac{K \gamma_{1+\epsilon}}{\log (v/u)} \right)^+ \phi Z_N } \right] \leq e^{-\phi(1+|\mu|\kappa/2)x},$$
for sufficiently large $x$, where $\phi = \frac{\epsilon}{(1-\delta)x + |\mu|w/2} \log \left( (1-\delta)x^\nu+|\mu|/(2\gamma_{1+\epsilon}) \right)$. The last step is to note that
\begin{align*}
-\frac{(1+|\mu|\kappa/2)\epsilon x}{(1-\delta)x + |\mu|w/2} \log \left( (1-\delta)x^\nu+|\mu|/(2\gamma_{1+\epsilon}) \right) &=  -\frac{(1+|\mu|\kappa/2)\epsilon}{1-\delta } \log \left( (1-\delta)x^\nu  \right) + o(1) \\
&\leq -(1+|\mu|\kappa/2)\epsilon \log x^\nu + O(1) \\
&= -( \beta+\epsilon+\epsilon\nu)\log x + O(1)
\end{align*}
as $x \to \infty$, which implies that \eqref{eq:LastNegBd} is $o( x^{-\beta-\epsilon})$. This completes the proof. 
\end{proof}

\section{The lower bound} \label{S.LowerBound}

We give in this section lower bounds for the tail distribution of the randomly weighted and randomly stopped sum. The idea of the proof is to split the probability $P(S_N > x)$ into several different probabilities, similarly to what was done for the upper bounds, and just keep those that determine the asymptotics. The first lemma is a preliminary step for Lemma \ref{L.MaxLowerBound}, and the main lower bounds are given in Lemmas~\ref{L.MaxLowerBound} and \ref{L.ZLarge}. 

\begin{lem} \label{L.MaxAsym}
Suppose the $\{X_i\}$ and the vector $(N, C_1, C_2, \dots)$ satisfy Assumption \ref{A.Main} with $\gamma_{1+\epsilon} = ||X_1||_{1+\epsilon} < \infty$ for some $\epsilon > 0$. Let $0 < \nu < 1$, $w = x^{1-\nu}/\gamma_{1+\epsilon}$, $y = x/\log x$ and $\delta > 0$. Then, there exist constants $K, x_0 > 0$ such that for all $x \geq x_0$, 
\begin{align*}
&P\left(J_N((1+\delta)x) = 1, \, Z_N \leq y, \, I_N(w) = 0  \right) \geq E\left[   \Indicator \left( Z_N \leq y, I_N(w) = 0 \right) \sum_{i=1}^N \overline{F}((1+\delta)x/ C_i)   \right]  \\
&\hspace{5mm} -   \frac{K}{x^{\nu\epsilon} \log x} E\left[ \Indicator\left( I_N(w) = 0 \right) \sum_{i=1}^N \overline{F}(x/ C_i)   \right] . 
\end{align*}
\end{lem}

\begin{proof}
Let $B_i = \{ C_i X_i > (1+\delta) x, \, \sup_{1\leq j < N+1, j\neq i} C_j X_j \leq (1+\delta) x, \, Z_N \leq y, \, I_N(w) = 0 \}$ and note that the $B_i$'s are disjoint. Therefore,
\begin{align*}
P\left(J_N( (1+\delta)x ) = 1, \, Z_N \leq y, \, I_N(w)=0 \right) &= E\left[  \Indicator\left( \bigcup_{i=1}^N B_i \right) \right] = E\left[  \sum_{i=1}^N \Indicator\left( B_i \right) \right].
\end{align*}
Let $\mathcal{F} = \sigma(N, C_1, C_2, \dots)$ and use the independence of the $\{X_i\}$ and $\mathcal{F}$ to  obtain
\begin{align*}
&E\left[  \sum_{i=1}^N \Indicator\left( B_i \right) \right] \\
&= E\left[  \Indicator \left( Z_N \leq y, \, I_N(w) = 0 \right) \sum_{i=1}^N E \left[ \left.  \Indicator\left( C_i X_i > (1+\delta) x \right) \Indicator \left( \sup_{1\leq j < N+1, j \neq i} C_j X_j \leq (1+\delta)x \right)  \right| \mathcal{F} \right] \right] \\
&= E\left[  \Indicator \left( Z_N \leq y, \, I_N(w) = 0 \right) \sum_{i=1}^N E \left[ \left.  \Indicator \left( C_i X_i > (1+\delta) x \right)  \right| \mathcal{F} \right]   \right] \\
&\hspace{5mm} - E\left[  \Indicator \left( Z_N \leq y, \, I_N(w) = 0 \right) \sum_{i=1}^N E \left[ \left.  \Indicator\left( C_i X_i > (1+\delta) x \right) \Indicator \left( \sup_{1\leq j < N+1, j \neq i} C_j X_j > (1+\delta)x \right)  \right| \mathcal{F} \right] \right] \\
&\geq  E\left[   \Indicator \left( Z_N \leq y, \, I_N(w) = 0 \right) \sum_{i=1}^N \overline{F}((1+\delta)x/ C_i)  \right] \\
&\hspace{5mm} - E\left[  \Indicator \left( Z_N \leq y, \, I_N(w) = 0 \right) \sum_{1 \leq i \neq j < N+1}  E \left[ \left.  \Indicator\left( C_i X_i > (1+\delta) x \right)   \Indicator \left( C_j X_j > (1+\delta)x \right)  \right| \mathcal{F} \right] \right],
\end{align*}
where in the last step we used the union bound. To bound the last expectation note that the conditional independence of the $\{C_i X_i\}$ given $\mathcal{F}$ gives
\begin{align*}
&E\left[  \Indicator \left( Z_N \leq y, \, I_N(w) = 0 \right) \sum_{1 \leq i \neq j < N+1} E \left[ \left.  \Indicator\left( C_i X_i > (1+\delta) x \right)   \Indicator \left( C_j X_j > (1+\delta)x \right)  \right| \mathcal{F} \right] \right] \\
&\leq E\left[  \Indicator \left( Z_N \leq y, \, I_N(w) = 0 \right) \left( \sum_{i=1}^N \overline{F}((1+\delta)x/ C_i)  \right)^2 \right].
\end{align*}
Now use the same arguments from Lemma \ref{L.SecondMax} to see that this last term is bounded from above by
\begin{align*}
&\frac{K}{x^{1+\epsilon}} y w^\epsilon E\left[ \Indicator\left( I_N(w) = 0 \right) \sum_{i=1}^N \overline{F}((1+\delta)x/ C_i)  \right] \leq \frac{K }{x^{\nu \epsilon} \log x}  E\left[ \Indicator\left( I_N(w) = 0 \right) \sum_{i=1}^N \overline{F}(x/ C_i)  \right].
\end{align*}
\end{proof}

\bigskip

The following result provides the first of the two terms determining the asymptotic behavior of $P(S_N > x)$, the one corresponding to the one-big-jump principle. Lemma \ref{L.ZLarge} will give the term corresponding to the case where the sum of the weights, $Z_N$, is large. 

\begin{lem} \label{L.MaxLowerBound}
Suppose the $\{X_i\}$ and the vector $(N, C_1, C_2, \dots)$ satisfy Assumption \ref{A.Main} with $\gamma_{1+\epsilon} = ||X_1||_{1+\epsilon} < \infty$ for some $\epsilon > 0$.  Let $0 < \nu < 1$, $w = x^{1-\nu}/\gamma_{1+\epsilon}$, $y = x/\log x$ and $0 < 1/\sqrt{\log x} \leq \delta < 1$. Then, for any $h > 0$, there exist constants $K, x_0 > 0$ such that for all $x \geq x_0$, 
\begin{align*}
&P\left(S_N > x, \, J_N((1+\delta)x) \geq 1, \,  L_N(y) = 0, \, Z_N \leq y, \, I_N(w) = 0  \right) \\
&\geq E\left[   \Indicator \left( Z_N \leq y, \, I_N(w) = 0 \right) \sum_{i=1}^N \overline{F}((1+\delta)x/ C_i)  \right] \\
&\hspace{5mm} - K\left( \frac{(\log x)^\epsilon}{x^{\nu\epsilon}} E\left[ \Indicator\left( I_N(w) = 0 \right) \sum_{i=1}^N \overline{F}(x/ C_i) \right] + \frac{1}{x^{h}} \right).
\end{align*}
\end{lem}

\begin{proof}
We start by noting that
\begin{align}
&P\left(S_N > x,\, J_N((1+\delta)x) \geq 1, \,  L_N(y) = 0, \, Z_N \leq y, \, I_N(w) = 0 \right) \notag \\
&=   P\left(J_N( (1+\delta)x) \geq 1, \, Z_N \leq y, \, I_N(w) = 0 \right)  \label{eq:MaxTerm} \\
&\hspace{5mm} -  P\left(J_N( (1+\delta)x) \geq 1, \, L_N(y) \geq 1, \, Z_N \leq y, \, I_N(w) = 0 \right)  \label{eq:MaxMinTerm} \\
&\hspace{5mm} -  P\left(S_N \leq x, \, J_N( (1+\delta)x) \geq 1, \, L_N(y) = 0, \, Z_N \leq y, \, I_N(w) = 0 \right)  .\label{eq:ExpSumTerm}
\end{align}
From Lemma \ref{L.MaxAsym} we obtain that \eqref{eq:MaxTerm} is greater than or equal to
$$E\left[   \Indicator \left( Z_N \leq y, \, I_N(w) = 0 \right) \sum_{i=1}^N \overline{F}((1+\delta)x/ C_i)  \right]  -   \frac{K}{x^{\nu\epsilon} \log x} E\left[ \Indicator\left( I_N(w) = 0 \right) \sum_{i=1}^N \overline{F}(x/ C_i)  \right] .$$

To bound \eqref{eq:MaxMinTerm} note that 
$$\left\{ J_N( (1+\delta)x) \geq 1, \, L_N(y) \geq 1 \right\} = \bigcup_{1 \leq i \neq j < N+1} \left\{C_i X_i > (1+\delta)x, \, C_j X_j < -y \right\}.$$
Now let $\mathcal{F} = \sigma(N, C_1,C_2, \dots)$ and use the union bound plus the conditional  independence of the $\{C_i X_i\}$ given $\mathcal{F}$ to obtain
\begin{align*}
&P\left( J_N( (1+\delta)x) \geq 1, \, L_N(y) \geq 1, \, Z_N \leq y, \, I_N(w) = 0 \right) \\
&= E\left[ \Indicator(Z_N \leq y, \, I_N(w) = 0) E\left[ \left. \Indicator\left( J_N( (1+\delta)x) \geq 1, \, L_N(y) \geq 1   \right) \right| \mathcal{F} \right] \right] \\
&\leq E\left[ \Indicator(Z_N \leq y, \, I_N(w) = 0) \sum_{1 \leq i \neq j < N+1} \overline{F}((1+\delta)x/ C_i) F(-y/C_j) \right] \\
&\leq E\left[ \Indicator(Z_N \leq y, \, I_N(w) = 0) \left( \sum_{i = 1}^N \overline{F}((1+\delta)x/ C_i) \right) \left( \sum_{j=1}^N F(-y/C_j) \right) \right] .
\end{align*}
The same arguments from Lemma \ref{L.SecondMax} now give that the last expectation is bounded by
\begin{align*}
\frac{K w^\epsilon }{y^\epsilon} E\left[ \Indicator\left( I_N(w) = 0 \right)  \sum_{i = 1}^N \overline{F}((1+\delta)x/ C_i)  \right] \leq \frac{K (\log x)^\epsilon }{x^{\nu\epsilon}} E\left[ \Indicator\left( I_N(w) = 0 \right) \sum_{i = 1}^N \overline{F}(x/ C_i)  \right] .
\end{align*}

Finally, to bound \eqref{eq:ExpSumTerm} note that 
\begin{align*}
&\left\{ S_N \leq x, \, J_N((1+\delta)x) \geq 1, \, L_N(y) = 0 \right\}  \\
&\subseteq \left\{  \sum_{i=1}^N C_i X_i 1(-y \leq C_i X_i \leq (1+\delta)x) \leq -\delta x \right\} \\
&\subseteq \left\{ \sum_{i=1}^N C_i X_i 1( C_i |X_i| \leq y) \leq -\delta x  \right\} \\
&\subseteq \left\{ \sum_{i=1}^N C_i |X_i| 1( C_i |X_i| \leq y ) \geq \delta x \right\},
\end{align*}
from where it follows that
\begin{align*}
&P\left(S_N \leq x, \, J_N( (1+\delta)x) \geq 1, \, L_N(y) = 0, \, Z_N \leq y, \, I_N(w) = 0 \right) \\
&\leq P\left( \sum_{i=1}^{N} C_i |X_i| \Indicator( C_i |X_i| \leq y) \geq \delta x, \, Z_N \leq y, \, I_N(w) = 0  \right) .
\end{align*}
Now apply Lemma \ref{L.StrongTrunc} with $X_i$ replaced by $|X_i|$ to obtain that 
$$P\left( \sum_{i=1}^{N} C_i |X_i| \Indicator( C_i |X_i| \leq y) \geq \delta x, \, Z_N \leq y, \, I_N(w) = 0 \right) = o\left( x^{-h} \right)$$
as $x \to \infty$ for all $h > 0$. 
 
Combining the bounds derived above for \eqref{eq:MaxTerm}, \eqref{eq:MaxMinTerm} and \eqref{eq:ExpSumTerm} gives the result. 
\end{proof}

\bigskip

\begin{lem} \label{L.ZLarge}
Suppose the $\{X_i\}$ and the vector $(N, C_1, C_2, \dots)$ satisfy Assumption \ref{A.Main} with $\gamma_{1+\epsilon} = ||X_1||_{1+\epsilon} < \infty$ for some $\epsilon > 0$. In addition assume that $Z_N < \infty$ a.s. and $E\left[ \sum_{i=1}^N C_i^{\beta+\epsilon} \right] < \infty$ for some $\beta > 0$.  Let $\nu = \epsilon/(2(\beta+\epsilon))$, $w = x^{1-\nu}/\gamma_{1+\epsilon}$ and $0 < 1/\sqrt{\log x} \leq \delta \leq 1/2$. Then, there exist constants $K, x_0 > 0$ such that for all $x \geq x_0$, 
\begin{align*}
&P\left(S_N > x, \, Z_N > (1+\delta)x/\mu, \, I_N(w) = 0  \right) \\
&\geq P(Z_N > (1+\delta)x/\mu) - K e^{-\epsilon\nu\sqrt{\log x}} P(Z_N > x/\mu) - \frac{K}{x^{\beta+\epsilon/2}} .
\end{align*}
\end{lem}

\begin{proof}
We start by noting that
\begin{align*}
&P\left(S_N > x, \, Z_N > (1+\delta)x/\mu, \, I_N(w) = 0  \right) \\
&\geq P\left(Z_N > (1+\delta)x/\mu  \right) - P\left(S_N \leq x, \, Z_N > (1+\delta)x/\mu, \, I_N(w) = 0  \right) - P\left(I_N(w) \geq 1  \right).
\end{align*}
From \eqref{eq:CMax} we obtain
$$P( I_N(w) \geq 1 ) \leq K x^{-\beta-\epsilon/2}.$$
Now let $\overline{X}_i = \mu - X_i$, $\widehat{X}_i = \mu/2 - X_i$,  $\overline{S}_N = \sum_{i=1}^N C_i \overline{X}_i$, and $\widehat{S}_N =  \sum_{i=1}^N C_i \widehat{X}_i$. Note that
\begin{align}
&P\left(S_N \leq x, \, Z_N > (1+\delta) x/\mu, \, I_N(w) = 0 \right) \notag \\
&= P\left( \overline{S}_N  \geq \mu Z_N -x, \, (1+\delta)x/\mu < Z_N \leq 4x/\mu, \, I_N(w) = 0 \right) \label{eq:MeanZero} \\
&\hspace{5mm} + P\left( \widehat{S}_N  \geq \mu Z_N/2 -x, \, Z_N > 4x/\mu, \, I_N(w) = 0 \right). \label{eq:NegLargeZ}
\end{align}
To analyze \eqref{eq:MeanZero} define $\overline{J}_N(t) = \#\{1 \leq i < N+1: C_i \overline{X}_i > t\}$ and note that \eqref{eq:MeanZero} is bounded by
\begin{align}
&P\left( \overline{S}_N  \geq \delta x, \, (1+\delta)x/\mu < Z_N \leq 4x/\mu, \, I_N(w) = 0 \right) \notag \\
&\leq P\left( \overline{S}_N  \geq \delta x, \, \overline{J}_N(x) = 0, \, (1+\delta)x/\mu < Z_N \leq 4x/\mu, \, I_N(w) = 0 \right) \label{eq:NegTruncSum} \\
&\hspace{5mm} + P\left( \overline{J}_N(x) \geq 1, \, (1+\delta)x/\mu < Z_N \leq 4x/\mu, \, I_N(w) = 0 \right) .\label{eq:NegBigJump}
\end{align}
By Lemma \ref{L.TruncBound} with $v = x$, $u = x^{1-\nu}$ and $A = ((1+\delta)x/\mu, 4x/\mu]$, \eqref{eq:NegTruncSum} is bounded by
\begin{align*}
&P\left( \sum_{i=1}^N C_i \overline{X}_i \Indicator(C_i \overline{X}_i \leq x) \geq \delta x, \, (1+\delta)x/\mu < Z_N \leq 4x/\mu, \, I_N(w) = 0 \right) \\
&\leq E\left[ \Indicator( (1+\delta)x/\mu < Z_N \leq 4x/\mu) e^{-\theta \delta x + \frac{K ||\overline{X}_1||_{1+\epsilon}}{\log(x^\nu)} \theta Z_N } \right] \\
&\leq e^{-\theta \delta x + \frac{4K ||\overline{X}_1||_{1+\epsilon}}{\mu \log(x^\nu)} \theta x } P(Z_N > x/\mu) \\
&\leq K e^{-\epsilon\nu\sqrt{\log x}} P(Z_N > x/\mu),
\end{align*}
where $\theta = \frac{\epsilon}{x} \log (x^\nu)$. To analyze \eqref{eq:NegBigJump} let $\mathcal{F} = \sigma(N, C_1, C_2, \dots)$ and use the union bound to see that it is bounded by
\begin{align}
&E\left[ \Indicator( (1+\delta)x/\mu < Z_N \leq 4x/\mu, \, I_N(w) = 0) E\left[ \left. \Indicator(\overline{J}_N(x) \geq 1) \right| \mathcal{F} \right] \right] \notag \\
&\leq E\left[ \Indicator( (1+\delta)x/\mu < Z_N \leq 4x/\mu, \, I_N(w) = 0) \sum_{i=1}^N E\left[ \left. \Indicator(C_i \overline{X}_i > x) \right| C_i \right] \right] \notag \\
&\leq \frac{E[|\overline{X}_1|^{1+\epsilon}]}{x^{1+\epsilon}}  E\left[ \Indicator( (1+\delta)x/\mu < Z_N \leq 4x/\mu, \, I_N(w) = 0) \sum_{i=1}^N C_i^{1+\epsilon} \right] \label{eq:intermediateStep} \\
&\leq \frac{K w^\epsilon}{x^{1+\epsilon}} E\left[ \Indicator( (1+\delta)x/\mu < Z_N \leq 4x/\mu) Z_N \right]  \notag \\
&\leq \frac{K}{x^{\nu\epsilon}} P(Z_N > x/\mu). \notag
\end{align}
Now, to analyze \eqref{eq:NegLargeZ} define $\widehat{J}_N(t) = \#\{1 \leq i < N+1: C_i \widehat{X}_i > t\}$ and split the probability into
\begin{align}
&P\left( \widehat{S}_N  \geq \mu Z_N/2 -x, \, Z_N > 4x/\mu, \, I_N(w) = 0, \, \widehat{J}_N(\mu Z_N) = 0 \right) \notag \\
&\hspace{5mm} + P\left( \widehat{S}_N  \geq \mu Z_N/2 -x, \, Z_N > 4x/\mu, \, I_N(w) = 0, \, \widehat{J}_N(\mu Z_N) \geq 1 \right) \notag \\
&\leq P\left( \sum_{i=1}^N C_i \widehat{X}_i \Indicator( C_i \widehat{X}_i \leq \mu Z_N) \geq \mu Z_N/2-x, \, Z_N > 4x/\mu, \, I_N(w) = 0 \right) \label{eq:LargeZTrunc} \\
&\hspace{5mm} + P\left( Z_N > 4x/\mu, \, I_N(w) = 0, \, \widehat{J}_N(\mu Z_N) \geq 1 \right). \label{eq:LargeZBigJump}
\end{align}
The same steps used to derive \eqref{eq:intermediateStep} give that \eqref{eq:LargeZBigJump} is bounded by
\begin{align*}
&E[|\widehat{X}_1|^{1+\epsilon}] E\left[ \Indicator(Z_N > 4x/\mu, \, I_N(w) = 0) (\mu Z_N)^{-1-\epsilon} \sum_{i=1}^N C_i^{1+\epsilon} \right] \\
&\leq K w^\epsilon E\left[ \Indicator(Z_N > 4x/\mu) Z_N^{-\epsilon} \right] \leq \frac{K}{x^{\nu\epsilon}} P(Z_N > 4x/\mu). 
\end{align*}
Finally, to bound \eqref{eq:LargeZTrunc} we can repeat the proof of Lemma \ref{L.TruncBound}, with the difference that $Z_N$ now appears in the truncation and the level to be exceeded. Set $v = \mu Z_N$, $u = x^{1-\nu}$, $z = \mu Z_N/2-x$, $\Theta = \frac{\epsilon}{\mu Z_N} \log ( x^{-1+\nu} \mu Z_N)$, and note that on the set $\{Z_N > 4x/\mu\}$ we have $1/v \leq \Theta \leq 1/u$ for sufficiently large $x$, as required. Now, the same proof of Lemma \ref{L.TruncBound} gives that \eqref{eq:LargeZTrunc} is bounded, for sufficiently large $x$, by
\begin{align*}
E\left[ \Indicator(Z_N > 4x/\mu) e^{-\Theta (\mu Z_N/2-x) + \left( -\mu/2 + \frac{K ||\widehat{X}_1||_{1+\epsilon}}{\log (x^{-1+\nu} \mu Z_N)} \right)^+ \Theta Z_N} \right] &\leq E\left[ \Indicator(Z_N > 4x/\mu) e^{-\Theta (\mu Z_N/2-x) } \right] .
\end{align*}
Now note that on $\{Z_N > 4x/\mu\}$ we have
$$-\Theta (\mu Z_N/2-x) = - \frac{\epsilon(\mu Z_N/2-x)}{\mu Z_N} \log( x^{-1+\nu} \mu Z_N) \leq -\frac{\epsilon}{4} \log( 4x^\nu),$$
which shows that \eqref{eq:LargeZTrunc} is bounded by $K x^{-\epsilon\nu/4} P(Z_N > 4x/\mu)$. This completes the proof. 
\end{proof}

\section{Proofs of the main theorems} \label{S.MainProofs}

In this section we give the proofs of the theorems in Section \ref{S.Main}. We start by stating two  preliminary lemmas. Lemma \ref{L.Potter} is included only for completeness since part (a) is a direct consequence of the Representation Theorem for the $OR$ class, Theorem 2.2.7 in \cite{BiGoTe1987}, and part (b) is contained in Theorem 2.3 in \cite{Cline_94}. 

\begin{lem} \label{L.Potter}
Suppose that $\overline{F} \in OR$ with Matuszewska indices $0 < \alpha_f \leq \beta_f < \infty$. Then, for any $\epsilon > 0$,
\begin{enumerate} \renewcommand{\labelenumi}{\alph{enumi})}
\item there exists $x_0 > 0$ such that $\overline{F}(x) \geq x^{-\beta_f - \epsilon}$ for all $x \geq x_0$.
\item there exist $x_0 > 0$ and $M < \infty$ such that for all $\lambda > 1$ and $x \geq x_0$,
$$\frac{1}{M} \lambda^{-\beta_f -\epsilon} \leq \frac{\overline{F}(\lambda x)}{\overline{F}(x)} \leq M \lambda^{-\alpha_f + \epsilon}.$$
\end{enumerate}
\end{lem}

\bigskip

The second preliminary lemma below establishes the one-big-jump asymptotics for the random weighted sum using the properties of the IR class.

\begin{lem} \label{L.SummandsAsym}
Suppose the $\{X_i\}$ and the vector $(N, C_1, C_2, \dots)$ satisfy Assumption \ref{A.Main}. Assume further that $\overline{F} \in IR$ and has Matuszewska indices $0 < \alpha_f \leq \beta_f < \infty$, and that $Z_N < \infty$ a.s., $E\left[ \sum_{i=1}^N C_i^{\alpha_f-\epsilon} \right] < \infty$ and $E\left[ \sum_{i=1}^N C_i^{\beta_f+\epsilon} \right] < \infty$ for some $0 < \epsilon < \alpha_f$.  If $E[N] < \infty$ then the condition $E\left[ \sum_{i=1}^N C_i^{\alpha_f-\epsilon} \right] < \infty$ can be dropped. Let $\nu = \epsilon/(2(\beta_f+\epsilon))$, $\gamma > 0$, $w = x^{1-\nu}/\gamma$, $y = x/\log x$ and $\delta = 1/\sqrt{\log x}$, then, as $x \to \infty$,
\begin{align*}
\mathbf{R} \triangleq E\left[ \sum_{i=1}^N \overline{F}(x/ C_i)   \right] &\sim
E\left[ \Indicator\left( I_N(w) = 0 \right) \sum_{i=1}^N \overline{F}((1-\delta)x/ C_i) \right]  \triangleq \mathbf{U} \\
&\sim E\left[   \Indicator \left( Z_N \leq y, \, I_N(w) = 0 \right) \sum_{i=1}^N \overline{F}((1+\delta)x/ C_i)  \right]  \triangleq \mathbf{L}.
\end{align*}
\end{lem}

\begin{proof}
We start with the upper bounds, 
\begin{align*}
\mathbf{U} &\leq  E\left[ \Indicator\left( I_N(w) = 0 \right) \sup_{1\leq j < N+1} \frac{\overline{F}((1-\delta)x/ C_j)}{\overline{F}(x/ C_j)} \sum_{i=1}^N \overline{F}(x/ C_i)  \right] \\
&\leq  \sup_{t \geq x/w} \frac{\overline{F}((1-\delta)t)}{\overline{F}(t)} \, \mathbf{R},
\end{align*}
and $\mathbf{L} \leq \mathbf{R}$. It follows that
$$\limsup_{x \to \infty} \frac{\mathbf{L}}{\mathbf{R}} \leq \limsup_{x \to \infty} \frac{\mathbf{U}}{\mathbf{R}} \leq 1.$$

Now, for the lower bounds we have
\begin{align*}
\mathbf{U} \geq \mathbf{R} - E\left[ \Indicator \left(I_N(w) \geq 1\right) \sum_{i=1}^N \overline{F}(x/ C_i)  \right]
\end{align*}
and
\begin{align*}
\mathbf{L} &\geq E\left[ \Indicator \left( I_N(w) = 0 \right) \inf_{1\leq j < N+1} \frac{\overline{F}((1+\delta)x/ C_j)}{\overline{F}(x/ C_j)} \sum_{i=1}^N \overline{F}(x/ C_i)   \right] \\
&\hspace{5mm} - E\left[ \Indicator \left( Z_N > y , \, I_N(w) = 0\right) \sum_{i=1}^N \overline{F}(x/ C_i)  \right] \\
&\geq \inf_{t \geq x/w} \frac{\overline{F}((1+\delta)t) }{\overline{F}(t)} \left( \mathbf{R} - E\left[ \Indicator\left(I_N(w) \geq 1\right) \sum_{i=1}^N \overline{F}(x/ C_i) \right] \right) \\
&\hspace{5mm} - E\left[ \Indicator\left( Z_N > y, \, I_N(w) = 0 \right) \sum_{i=1}^N \overline{F}(x/ C_i) \right].
\end{align*}
It remains to show that
$$\lim_{x \to \infty} \frac{E\left[ \Indicator\left(I_N(w) \geq 1\right) \sum_{i=1}^N \overline{F}(x/ C_i)   \right] + E\left[ \Indicator\left(Z_N > y, I_N(w) = 0\right) \sum_{i=1}^N \overline{F}(x/ C_i)  \right]
}{\mathbf{R}} = 0 .$$
To obtain a lower bound for $\mathbf{R}$ we use Lemma \ref{L.Potter} (b) and Fatou's lemma as follows, 
\begin{equation} \label{eq:LowerR}
\liminf_{x \to \infty} \frac{\mathbf{R}}{\overline{F}(x)} \geq E\left[   \sum_{i=1}^N \liminf_{x \to \infty} \frac{\overline{F}(x/ C_i)}{\overline{F}(x)} \right]  \geq K E\left[ \sum_{i=1}^N C_i^{\alpha_f-\epsilon} \wedge C_i^{\beta_f+\epsilon} \right] > 0.
\end{equation}
Thus, it suffices to prove that
\begin{equation} \label{eq:TwoLimits}
\lim_{x \to \infty} E\left[ \Indicator\left(I_N(w) \geq 1\right) \sum_{i=1}^N \frac{\overline{F}(x/ C_i)}{\overline{F}(x)}  \right] =  \lim_{x \to \infty} E\left[ \Indicator\left(Z_N > y, I_N(w) = 0\right) \sum_{i=1}^N \frac{\overline{F}(x/ C_i)}{\overline{F}(x)}   \right] = 0 .
\end{equation}
We analyze the second limit by noting that by Lemma \ref{L.Potter} (b), we have that for all sufficiently large $x$,  
\begin{align*}
E\left[ \Indicator\left(Z_N > y, I_N(w) = 0\right) \sum_{i=1}^N \frac{\overline{F}(x/ C_i)}{\overline{F}(x)} \right] &\leq K E\left[ \sum_{i=1}^N C_i^{\alpha_f-\epsilon} \vee C_i^{\beta_f+\epsilon}  \right]  \\
&\leq K E\left[ \sum_{i=1}^N C_i^{\alpha_f-\epsilon} \right]  + K E\left[ \sum_{i=1}^N C_i^{\beta_f+\epsilon} \right] < \infty,
\end{align*}
so by dominated convergence,
\begin{align}
&\limsup_{x \to \infty} E\left[ \Indicator\left(Z_N > y, I_N(w) = 0\right) \sum_{i=1}^N \frac{\overline{F}(x/ C_i)}{\overline{F}(x)}   \right] \notag \\
&\leq  K E\left[ \limsup_{x \to \infty} \Indicator\left(Z_N > y \right) \sum_{i=1}^N C_i^{\alpha_f-\epsilon} \vee C_i^{\beta_f + \epsilon}  \right] = 0. \label{eq:MomentCondition}
\end{align}
For the first limit in \eqref{eq:TwoLimits} we first split the expectation to obtain
\begin{align}
E\left[ \Indicator\left(I_N(w) \geq 1\right) \sum_{i=1}^N \frac{\overline{F}(x/ C_i)}{\overline{F}(x)}  \right] &\leq E\left[ \Indicator\left(I_N(w) \geq 1\right) \sum_{i=1}^N \frac{\overline{F}(x/ C_i)}{\overline{F}(x)}  \Indicator(C_i \leq w)  \right] \label{eq:LargeBdC} \\
&\hspace{5mm} + E\left[ \Indicator\left(I_N(w) \geq 1\right) \sum_{i=1}^N \frac{\Indicator(C_i > w)}{\overline{F}(x)}   \right] .\notag
\end{align}
Dominated convergence again gives 
\begin{align*}
&\limsup_{x \to \infty} E\left[ \Indicator\left(I_N(w) \geq 1\right) \sum_{i=1}^N\frac{\overline{F}(x/ C_i)}{\overline{F}(x)}  \Indicator(C_i \leq w)  \right] \\
&\leq K E\left[ \limsup_{x \to \infty} \Indicator\left(I_N(w) \geq 1\right) \sum_{i=1}^N C_i^{\alpha_f-\epsilon} \vee C_i^{\beta_f+\epsilon}  \right] = 0.
\end{align*}
Finally, to bound \eqref{eq:LargeBdC} note that by \eqref{eq:CMax},
$$E\left[ \Indicator\left(I_N(w) \geq 1\right) \sum_{i=1}^N \Indicator(C_i > w)   \right] \leq E\left[ \sum_{i=1}^N \Indicator\left( C_i > w \right) \right] \leq \frac{K}{x^{\beta_f+\epsilon/2}}.$$
The observation that by Lemma \ref{L.Potter} (a) $\lim_{x \to \infty} x^{\beta_f+\epsilon/2} \overline{F}(x) = \infty$ completes the proof. 
\end{proof}

{\bf Remark:} The proof given above requires that $E\left[ \sum_{i=1}^N C_i^{\alpha_f-\epsilon} \vee C_i^{\beta_f + \epsilon} \right] < \infty$ to derive \eqref{eq:MomentCondition}, which is clearly implied by the two conditions $E\left[ \sum_{i=1}^N C_i^{\alpha_f - \epsilon} \right] < \infty$ and $E\left[ \sum_{i=1}^N C_i^{\beta_f+\epsilon} \right] < \infty$. To see that the first condition can be dropped when $E[N] < \infty$ note that 
$$E\left[ \sum_{i=1}^N C_i^{\alpha_f-\epsilon} \vee C_i^{\beta_f+\epsilon} \right] \leq E[N] + E\left[ \sum_{i=1}^N C_i^{\beta_f+\epsilon} \right] < \infty.$$

We are now ready to prove the main theorems from Section \ref{S.Main}. The first result corresponds to the setting where the asymptotic behavior of both $P(M_N > x)$ and $P(S_N > x)$ is determined by the one-big-jump principle. 

\begin{proof}[Proof of Theorem \ref{T.SummandsDominate}]
Let $\alpha = \alpha_f$, $\beta = \beta_f$, and $\mathbf{R} = E\left[ \sum_{i=1}^N \overline{F}(x/ C_i) \right]$. Note that by \eqref{eq:LowerR} we have that $\mathbf{R} \geq K \overline{F}(x)$, and by Lemma \ref{L.Potter} (a) we have that $\lim_{x \to \infty} x^{\beta+h} \overline{F}(x) = \infty$ for any $h > 0$, from where it follows that 
\begin{equation} \label{eq:littleO1}
K x^{-\beta-\epsilon/2} = o(\mathbf{R})
\end{equation}
as $x \to \infty$. Let $\nu = \epsilon/(2(\beta+\epsilon))$, $w = x^{1-\nu}/\gamma_{1+\epsilon}$, $y = x/\log x$, and $\delta = 1/\sqrt{\log x}$. Then, from Lemmas \ref{L.MaxLowerBound} and \ref{L.SummandsAsym} we obtain, for all three cases, that
$$\liminf_{x \to \infty} \frac{P(M_N > x)}{\mathbf{R}} \geq \liminf_{x \to \infty} \frac{P(S_N > x)}{\mathbf{R}} \geq 1.$$

For the upper bound we first note that by Lemma \ref{L.Potter} (b), 
\begin{align}
&\frac{(\log x)^{1+\epsilon} }{ x^{\epsilon\nu} } E\left[  \Indicator\left( I_N(w) = 0 \right) \sum_{i = 1}^N \overline{F}(y/ C_i)  \right] \notag  \\
&\leq \frac{(\log x)^{1+\epsilon} }{ x^{\epsilon\nu} } E\left[  \Indicator\left( I_N(w) = 0 \right)  \sum_{i = 1}^N K (y/x)^{-\beta-\epsilon} \overline{F}(x/ C_i)  \right] \notag \\
&\leq  \frac{K(\log x)^{\beta+1+2\epsilon} }{ x^{\epsilon\nu} } \cdot \mathbf{R} = o\left(\mathbf{R} \right), \label{eq:littleO2}
\end{align}
for all sufficiently large $x$. We split the rest of the analysis of the upper bounds into the three different cases.

{\bf Case 1:} $\mu < 0$. It follows from Proposition \ref{P.UpperBounds} (b), Lemma \ref{L.SummandsAsym}, and relations \eqref{eq:littleO1} and \eqref{eq:littleO2}, that
$$\limsup_{x\to \infty} \frac{P(M_N > x)}{\mathbf{R}} \leq 1.$$

{\bf Case 2:} $\mu = 0$ and $P(Z_N > x) = O\left(\overline{F}(x) \right)$. We use Proposition \ref{P.UpperBounds} (a),  Lemma \ref{L.SummandsAsym}, and relations \eqref{eq:littleO1} and \eqref{eq:littleO2} to obtain
$$\limsup_{x\to \infty} \frac{P(M_N > x)}{\mathbf{R}} \leq 1 + \limsup_{x \to \infty} \frac{P(\delta Z_N > x) + K x^{-\epsilon\nu/2} P(Z_N > y) }{\mathbf{R}}.$$
To see that the last limit is zero use Lemma \ref{L.Potter} to obtain
\begin{align*}
 \limsup_{x \to \infty} \frac{P(\delta Z_N > x) + K x^{-\epsilon\nu/2} P(Z_N > y) }{\mathbf{R}} &\leq K  \limsup_{x \to \infty} \frac{\overline{F}(x/ \delta)  + x^{-\epsilon\nu/2} \overline{F}(y) }{\overline{F}(x)} \\
 &\leq K \limsup_{x \to \infty} \left( \delta^{\alpha-\epsilon} + x^{-\epsilon \nu/2} (x/y)^{\beta+\epsilon} \right) \\
 &= K \limsup_{x \to \infty} \left( \frac{1}{(\log x)^{(\alpha-\epsilon)/2}} + \frac{(\log x)^{\beta+\epsilon}}{x^{\epsilon \nu/2} } \right) = 0. 
\end{align*}

{\bf Case 3:} $\mu > 0$ and $P(Z_N > x) = o\left( \overline{F}(x) \right)$. We use Proposition \ref{P.UpperBounds} (a), Lemma \ref{L.SummandsAsym}, and relations \eqref{eq:littleO1} and \eqref{eq:littleO2} to obtain
\begin{align*}
\limsup_{x\to \infty} \frac{P(M_N > x)}{\mathbf{R}} &\leq 1 + \limsup_{x \to \infty} \frac{P((\mu+\delta) Z_N > x) + K  x^{-\epsilon\nu/2} P(Z_N > y) + K e^{-\frac{\epsilon\nu\sqrt{\log x}}{\mu}} P(Z_N > x/(2\mu))}{\mathbf{R}} \\
&\leq 1 + K \limsup_{x \to \infty} \frac{ P((\mu+\delta) Z_N > x) + e^{-\frac{\epsilon\nu\sqrt{\log x}}{\mu}} P(Z_N > y) }{\overline{F}(x)}.
\end{align*}
For the first summand in the limit we use Lemma \ref{L.Potter} to see that 
\begin{align*}
\limsup_{x \to \infty} \frac{P((\mu+\delta) Z_N > x)}{\overline{F}(x)} &\leq \limsup_{x \to \infty} \frac{P((\mu+\delta) Z_N > x)}{\overline{F}(x/(\mu+\delta))} \cdot  \limsup_{x \to \infty} \frac{\overline{F}(x/(\mu+\delta))}{\overline{F}(x)} \\
&\leq K  \limsup_{x \to \infty} \frac{P(Z_N > x/(\mu+\delta))}{\overline{F}(x/(\mu+\delta))}  = 0.
\end{align*}
For the second limit we use Lemma \ref{L.Potter} again as follows:
\begin{align*}
\limsup_{x \to \infty} \frac{e^{-\frac{\epsilon\nu\sqrt{\log x}}{\mu}} P(Z_N > y) }{\overline{F}(x)} &\leq \limsup_{x \to \infty} \frac{P(Z_N > y)}{\overline{F}(y)} \cdot \frac{e^{-\frac{\epsilon\nu\sqrt{\log x}}{\mu}} \overline{F}(y)}{\overline{F}(x)} \\
&\leq \limsup_{x \to \infty} \frac{P(Z_N > y)}{\overline{F}(y)} \cdot K \limsup_{x \to \infty} e^{-\frac{\epsilon\nu\sqrt{\log x}}{\mu}} (\log x)^{\beta+\epsilon} = 0. 
\end{align*}
\end{proof}

\bigskip

The next proof corresponds to the setting where the asymptotic behavior of $P(M_N > x)$ and $P(S_N > x)$ is determined by both the one-big-jump principle and the tail behavior of $Z_N$. 

\begin{proof}[Proof of Theorem \ref{T.Both}]
Let $\alpha = \alpha_f$, $\beta = \beta_f$, and $\mathbf{R} = E\left[ \sum_{i=1}^N \overline{F}(x/C_i) \right]$. Note that by \eqref{eq:LowerR} we have that $\mathbf{R} \geq K \overline{F}(x)$, and by Lemma \ref{L.Potter} (a) we have that $\lim_{x \to \infty} x^{\beta+h} \overline{F}(x) = \infty$ for any $h > 0$, from where it follows that 
$K x^{-\beta-\epsilon/2} = o(\mathbf{R})$ as $x \to \infty$. Let $\nu = \epsilon/(2(\beta+\epsilon))$, $w = x^{1-\nu}/\gamma_{1+\epsilon}$, $y = x/\log x$, and $\delta = 1/\sqrt{\log x}$. 

Note that since $Z_N$ is $IR$,  $P((\mu+\delta) Z_N > x) \sim P(Z_N > x/\mu)$ as $x \to \infty$. Also, since $IR \subset OR$, it holds that
$$\limsup_{x \to \infty} \frac{P(Z_N > x/(2\mu))}{P(Z_N > x/\mu)} < \infty.$$
Moreover, if we let $0 \leq \beta_g < \infty$ be the lower Matuszewska index of $\overline{G}(x) = P(Z_N > x)$, then Lemma~\ref{L.Potter} (b) gives
$$\limsup_{x \to \infty} \frac{x^{-\epsilon\nu/2} P(Z_N > y)}{P(Z_N > x/\mu)} \leq K \limsup_{x \to \infty} x^{-\epsilon\nu/2} (\log x/\mu)^{\beta_g+\epsilon} = 0.$$
These observations combined with Proposition \ref{P.UpperBounds} (a), Lemma \ref{L.SummandsAsym}, and relations \eqref{eq:littleO1} and \eqref{eq:littleO2}, give
\begin{align*}
\limsup_{x\to \infty} \frac{P(M_N > x)}{\mathbf{R} + P(Z_N > x/\mu)} &\leq 1 + K \limsup_{x \to \infty} \frac{x^{-\epsilon\nu/2} P(Z_N > y) + e^{-\frac{\epsilon\nu\sqrt{\log x}}{\mu}} P(Z_N > x/(2\mu))  }{ P(Z_N > x/\mu)} = 1.
\end{align*}

For the lower bound we use $P(Z_N > (1+\delta)x/\mu) \sim P(Z_N > x/\mu)$, Lemmas \ref{L.MaxLowerBound}, \ref{L.ZLarge} and \ref{L.SummandsAsym}, and relations \eqref{eq:littleO1} and \eqref{eq:littleO2} to obtain
$$\liminf_{x \to \infty} \frac{P(M_N > x)}{\mathbf{R} + P(Z_N > x/\mu)} \geq \liminf_{x \to \infty} \frac{P(S_N > x)}{\mathbf{R}+ P(Z_N > x/\mu)} = 1 .$$
This completes the proof. 
\end{proof}

\bigskip

The last proof corresponds to the setting where the tail behavior of $P(M_N > x)$ and $P(S_N > x)$ is solely determined by the sum of the weights, $Z_N$. 

\begin{proof}[Proof of Theorem \ref{T.ZDominates}]
Let $\alpha = \alpha_g$, $\beta = \beta_g$, $\nu = \epsilon/(2(\beta+\epsilon))$, $w = x^{1-\nu}/\gamma_{1+\epsilon}$, $y = x/\log x$, and $\delta = 1/\sqrt{\log x}$. Recall that $\overline{F}(x) = P(X_1 > x)$ and $\overline{G}(x) = P(Z_N > x)$. 

Note that since $\overline{G} \in IR$, then $P(Z_N > (1+\delta)x/\mu) \sim P((\mu+\delta) Z_N > x) \sim  P(Z_N > x/\mu)$ as $x \to \infty$.

We start with the upper bound, for which we use Proposition \ref{P.UpperBounds} (a) to obtain
\begin{align*}
\limsup_{x \to \infty} \frac{P(M_N > x)}{P(Z_N > x/\mu)} &\leq 1 + K \limsup_{x \to \infty} \frac{1}{P(Z_N > x/\mu)} \left\{ E\left[ \Indicator\left( I_N(w) = 0 \right) \sum_{i=1}^N \overline{F}((1-\delta)x/C_i) \right] \right. \\
&\hspace{20mm} \left. + \frac{(\log x)^{1+\epsilon}}{x^{\epsilon\nu}} E\left[ \Indicator\left( I_N(w) = 0 \right) \sum_{i=1}^N \overline{F}(y/C_i) \right]  + \frac{1}{x^{\beta+\epsilon/2}} \right. \\
&\hspace{12mm} \left. \phantom{\left[ \sum_i^N \right.} + x^{-\epsilon\nu/2} P(Z_N > y) +  e^{-\frac{\epsilon\nu\sqrt{\log x}}{\mu}} P(Z_N > x/(2\mu))  \right\}.
\end{align*}
Since the distribution of $Z_N$ belongs to $IR \subset OR$, then $\lim_{x \to \infty} x^{\beta+\epsilon/2} P(Z_N > x/\mu) = \infty$ by Lemma~\ref{L.Potter}~(a). Also, by the same arguments used in the proof of Theorem \ref{T.Both},
$$\limsup_{x \to \infty} \frac{ x^{-\epsilon\nu/2} P(Z_N > y) + e^{-\frac{\epsilon\nu\sqrt{\log x}}{\mu}} P(Z_N > x/(2\mu))   }{P(Z_N > x/\mu)} = 0.$$
For the two remaining terms we use Lemma \ref{L.Potter} to obtain, for sufficiently large $x$, 
\begin{align*}
&\frac{1}{\overline{G}(x/\mu)} \left\{ E\left[ \Indicator\left( I_N(w) = 0 \right) \sum_{i=1}^N \overline{F}((1-\delta)x/C_i)  \right]  + \frac{(\log x)^{1+\epsilon}}{x^{\epsilon\nu}} E\left[ \Indicator\left( I_N(w) = 0 \right) \sum_{i=1}^N \overline{F}(y/C_i) \right]    \right\} \\
&\leq \sup_{t \geq y/w} \frac{\overline{F}(t)}{\overline{G}(t)} \left\{ E\left[ \Indicator\left( I_N(w) = 0 \right) \sum_{i=1}^N \frac{\overline{G}((1-\delta)x/C_i)}{\overline{G}(x/\mu)}  \right] +  \frac{(\log x)^{1+\epsilon}}{x^{\epsilon\nu}} E\left[ \Indicator\left( I_N(w) = 0 \right) \sum_{i=1}^N \frac{\overline{G}(y/C_i)}{\overline{G}(x/\mu)} \right]  \right\} \\
&\leq  \sup_{t \geq y/w} \frac{\overline{F}(t)}{\overline{G}(t)} \left\{ E\left[ K  \sum_{i=1}^N \left( \frac{C_i}{(1-\delta)\mu} \right)^{\alpha-\epsilon} \vee \left( \frac{C_i}{(1-\delta)\mu } \right)^{\beta+\epsilon}     \right] \right. \\
&\hspace{35mm} \left. +  \frac{(\log x)^{1+\epsilon}}{x^{\epsilon\nu}} E\left[  K \sum_{i=1}^N \left( \frac{C_i x}{\mu y} \right)^{\alpha-\epsilon} \vee \left( \frac{C_i x}{\mu y} \right)^{\beta+\epsilon}     \right]  \right\} \\
&\leq K  \sup_{t \geq y/w} \frac{\overline{F}(t)}{\overline{G}(t)} \left\{ 1  +  \frac{(\log x)^{\beta+1+2\epsilon}}{x^{\epsilon\nu}} \right\} \leq K  \sup_{t \geq y/w} \frac{\overline{F}(t)}{\overline{G}(t)}  . 
\end{align*}
Since $\overline{F}(x) = o\left( \overline{G}(x) \right)$ as $x \to \infty$, we have $\limsup_{x \to \infty} \sup_{t \geq y/w} \frac{\overline{F}(t)}{\overline{G}(t)} = 0$. The expectation preceding the supremum is finite either if $E\left[ \sum_{i=1}^N C_i^{\alpha-\epsilon} \right] < \infty$ and $E\left[ \sum_{i=1}^N C_i^{\beta+\epsilon} \right] < \infty$, or if $E[N] < \infty$ and $E\left[ \sum_{i=1}^N C_i^{\beta+\epsilon} \right] < \infty$ (see the remark following the proof of Lemma \ref{L.SummandsAsym}).

For the lower bound we use $P(Z_N > (1+\delta)x/\mu) \sim P(Z_N > x/\mu)$ and Lemma \ref{L.ZLarge} to obtain 
\begin{align*}
\liminf_{x \to \infty} \frac{P(M_N > x)}{P(Z_N > x/\mu)} &\geq \liminf_{x \to \infty} \frac{P(S_N > x)}{P(Z_N > x/\mu)} \geq 1.
\end{align*}
This completes the proof. 
\end{proof}

\section*{Acknowledgements}
The author would like to thank an anonymous referee for his/her thorough reading of the paper and helpful comments. This work was supported by NSF Grant CMMI-1131053.

\bibliographystyle{plain}

\end{document}